\documentclass[11pt]{amsart}

\usepackage{tikz}
\usetikzlibrary{matrix,arrows}

\usepackage{mathrsfs}
\usepackage{amsmath, amsthm, amsfonts}
\usepackage[percent]{overpic}
\usetikzlibrary{matrix,arrows}

\usepackage{xypic}
\usepackage{color}
\usepackage{MnSymbol}
\usepackage{fullpage}
\usepackage{bbm}

\usepackage{thmtools}
\usepackage{thm-restate}

\usepackage{hyperref}

\usepackage{cleveref}

\declaretheorem[name=Theorem,numberwithin=section]{thm}

\newtheorem*{thm*}{Theorem}

\newtheorem{prop}[thm]{Proposition}

\theoremstyle{definition}
\newtheorem{defn}[thm]{Definition}

\theoremstyle{remark}

\newcommand{\PPP}{\mathcal{P}}

\newcommand{\RR}{\mathbb{R}}

\newcommand{\OO}{\mathcal{O}}
\newcommand{\ord}{\operatorname{ord}}

\newcommand{\ZZ}{\mathbb{Z}}

\newcommand{\PP}{\mathbb{P}}

\newcommand{\Pic}{\text{Pic}}

\newcommand{\Mgn}{\mathcal{M}_{g,n}}
\newcommand{\Mgb}{\overline{\mathcal{M}}_g}
\newcommand{\Mgnb}{\overline{\mathcal{M}}_{g,n}}

\newcommand{\rk}{\text{rk}}

\newcommand{\res}{\text{res}}

\newtheorem{Thm*}{Theorem*}
\theoremstyle{definition}

\title{On the effective cone of higher codimension cycles in $\overline{\mathcal{M}}_{g,n}$
}
\author{Scott Mullane\\
}
\date{\today}

\AtEndDocument{\bigskip{\footnotesize%
\textsc{Scott Mullane, Institut f\"{u}r Mathematik, Goethe-Universit\"{a}t Frankfurt, Robert-Mayer-Str. 6-8, 60325 Frankfurt am Main, Germany}\par
\textit{E-mail address}: \texttt{mullane@math.uni-frankfurt.de} \par
}}

\begin{document}
\thispagestyle{empty}

\maketitle

\begin{abstract}
We exhibit infinitely many extremal effective codimension-$k$ cycles in $\Mgnb$ in the cases
\begin{itemize}
\item
$g\geq 3, n\geq g-1$ and $k=2$,
\item
$g\geq 2$, $k\leq \min(n-g,g),$ and
\item
 $g=1$, $k\leq n-2$.
\end{itemize}
Hence in these cases the effective cone is not rational polyhedral. 
\end{abstract}

\setcounter{tocdepth}{2}


\section{Introduction}
The birational geometry of the moduli space of curves is broadly dictated by the effective cone of divisors, which has attracted much attention~\cite{HarrisMumford}\cite{EisenbudHarrisKodaira}\cite{Farkas23}\cite{FarkasPopa}\cite{ChenCoskun}\cite{Mullane3}. Though comparatively little is known, there has recently been growing interest in understanding finer aspects of the birational geometry encoded in the cones of higher codimension cycles~\cite{ChenCoskunH}\cite{FL1}\cite{FL2}\cite{ChenTarasca}. 
In this paper we use meromorphic differentials on curves to construct infinitely many cycles that form extremal rays of the effective cone of higher codimension cycles of $\Mgnb$ for fixed  $g$ and $n$. Hence in these cases we show that the effective cone is not rational polyhedral.

The meromorphic strata of canonical divisors of type $\kappa=(k_1,\dots,k_n)$, 
\begin{equation*}
\PPP(\kappa):=\left\{ [C,p_1,\dots,p_n]\in\Mgn\hspace{0.3cm}\big| \hspace{0.3cm}\sum k_ip_i\sim K_C   \right\}
\end{equation*}
for $\kappa$ a meromorphic partition of $2g-2$, form codimension $g$ subvarieties in $\Mgn$. When $g=1$ this condition corresponds to a condition in the group law on an elliptic curve. Chen and Coskun~\cite{ChenCoskun} showed that for $n\geq 3$, in infinitely many cases the closure produces rigid and extremal divisors in $\overline{\mathcal{M}}_{1,n}$ and hence the effective cone of divisors is not rational polyhedral in these cases. In the case that $g\geq2$, the author~\cite{Mullane3} showed the closure of infinitely many of these higher codimension cycles pushforward under the morphism forgetting marked points to give rigid and extremal divisors in $\Mgnb$ for $n\geq g+1$. Hence in these cases the effective cones of divisors are not rational polyhedral.

Chen and Coskun~\cite{ChenCoskunH} showed higher codimension boundary strata to be extremal in some cases and in the cases $\overline{\mathcal{M}}_{1,n}$ for $n\geq 5$ and  $\overline{\mathcal{M}}_{2,n}$ for $n\geq 2$ they used the infinitely many extremal divisors presented in~\cite{ChenCoskun} to produce infinitely many extremal codimension two cycles supported in the boundary of the moduli space. Schaffler~\cite{Schaffler} used the same strategy to obtain $315$ extremal codimension two cycles supported in the boundary of $\overline{\mathcal{M}}_{0,7}$ from the Keel-Vermeire divisors in $\overline{\mathcal{M}}_{0,6}$. From the interior of the moduli space, Chen and Coskun~\cite{ChenCoskunH} identified the closure of the locus of hyperelliptic curves with a marked Weierstrass point in  $\overline{\mathcal{M}}_{3,1}$ and the closure of the locus of hyperelliptic curves in  $\overline{\mathcal{M}}_{4}$ as extremal codimension two cycles. Chen and Tarasca~\cite{ChenTarasca} showed that for $1\leq n\leq 6$, marking $n$ Weierstrass points on a curve gave an extremal codimension $n$ cycle  in $\overline{\mathcal{M}}_{2,n}$. Blankers~\cite{Blankers} extended this to include marking any combination of Weierstrass points and pairs of points that are conjugate under the hyperelliptic involution. However, for any fixed $n$ this still only produced finitely many extremal higher codimension cycles coming from the interior of $\overline{\mathcal{M}}_{2,n}$. 

In this paper we use three methods to construct infinitely many extremal higher codimension cycles from the strata of canonical divisors. In \S\ref{exbdry} we use a gluing construction to obtain extremal codimension two cycles from extremal divisors in a moduli space of lower genus. In \S\ref{principal} and \S\ref{meromorphic} we use an inductive argument to give conditions on $\kappa$ for when the strata of canonical divisors $\PPP(\kappa)$ and pushfowards of this cycle forgetting marked points give extremal higher codimension cycles. In \S\ref{genusone} we restrict to genus $g=1$ and find conditions under which intersecting the strata of canonical divisors will result in extremal higher codimension cycles.

In \S\ref{exbdry} we use the extremal divisors from~\cite{Mullane3} to construct infinitely many codimension two cycles supported in the boundary of $\Mgnb$ for $g\geq 3$ and $n\geq g-1$ giving the following theorem.
\begin{restatable}{thm}{Bdry}
\label{thm:Bdry}
$\text{Eff}^2(\Mgnb)$ is not rational polyhedral for $g\geq 3$ and $n\geq g-1$.
\end{restatable}
We follow the general strategy of Chen and Coskun~\cite{ChenCoskunH}. Consider the gluing morphism 
\begin{equation*}
\alpha:\widehat{\Delta}_{1:\emptyset}=\overline{\mathcal{M}}_{g-1,n+1}\times \overline{\mathcal{M}}_{1,1}\longrightarrow \Mgnb
\end{equation*}
for $g\geq 3$, which glues a $[C,p_1,\dots,p_{n+1}]\in\overline{\mathcal{M}}_{g-1,n+1}$ to $[E,q]\in\overline{\mathcal{M}}_{1,1}$ by identifying $p_{n+1}$ with $q$ to form a node. 
Proposition~\ref{prod} shows pulling back extremal divisors on $\overline{\mathcal{M}}_{g-1,n+1}$ provides extremal divisors on $\widehat{\Delta}_{1:\emptyset}$. Further, the image of $\alpha$ is the locus contracted by the morphism $ps:\Mgnb\longrightarrow\Mgnb^{\text{ps}}$ that contracts unmarked elliptic tails to cusps, where $\Mgnb^{\text{ps}}$ is the alternate compactification of $\Mgn$ by pseudo-stable curves. In this situation Proposition~\ref{higherextremal} shows these cycles will pushforward to provide extremal codimension two cycles in $\Mgnb$ provided
\begin{equation*}
\alpha_*:A^1(\widehat{\Delta}_{1:\emptyset})=N^1(\widehat{\Delta}_{1:\emptyset})\longrightarrow N^2(\Mgnb)
\end{equation*}
is injective. Proposition~\ref{injective} shows this map to be injective in the cases considered by showing there are no nontrivial relations between the images of the generators of $N^1(\widehat{\Delta}_{1:\emptyset})$ through the use of test surfaces in $\Mgnb$ and pushing forward any possible such relation under forgetful morphisms and $ps$ that contracts unmarked elliptic tails.

For the rest of the paper we turn to cycles intersecting the interior of the moduli space. In \S\ref{principal} we show that pushforwards of the principal strata give rigid and extremal cycles in all codimensions.
\begin{restatable}{thm}{Principal}
\label{thm:Principal} The cycle $[\varphi_j{}_*\overline{\PPP}(1^{2g-2})]$ for $j=0,\dots,g-2$ is rigid and extremal in $\mbox{Eff}^{\hspace{0.1cm}g-j-1}(\overline{\mathcal{M}}_{g,2g-j-2})$, 
 where $\varphi_j:\overline{\mathcal{M}}_{g,2g-2}\longrightarrow\overline{\mathcal{M}}_{g,2g-j-2}$ forgets the last $j$ points.
\end{restatable}
This extends the result of Farkas and Verra~\cite{FarkasVerraU} on divisors, which is used as the base case in an inductive argument similar to that used by Chen and Tarasca~\cite{ChenTarasca}. This method is developed to more complicated situations in later sections. Assume the cycle $[\varphi_{j+1}{}_*\overline{\PPP}(1^{2g-2}) ]$ is rigid and extremal and let
\begin{equation*}
[\varphi_{j}{}_*\overline{\PPP}(1^{2g-2}) ]=\sum c_i[V_i]
\end{equation*}
be an effective decomposition with $c_i>0$ and $V_i$ irreducible codimension one subvarieties distinct from $\varphi_{j}{}_*\overline{\PPP}(1^{2g-2}) $. Pushing forward under the map $\pi_m:\overline{\mathcal{M}}_{g,2g-j-2}\longrightarrow \overline{\mathcal{M}}_{g,2g-j-3}$ forgetting the $m$th marked point for $m=1,\dots,2g-j-2$ gives the assumed extremal cycle $[\varphi_{j+1}{}_*\overline{\PPP}(1^{2g-2})]$ and hence there is some $l$ with 
\begin{equation*}
\pi_m{}_*[V_l]=k[\varphi_{j+1}{}_*\overline{\PPP}(1^{2g-2})]
\end{equation*} 
for $k>0$. But as $\varphi_{j+1}{}_*\PPP(1^{2g-2})$ is rigid, this implies that $V_l$ is supported in $\pi_m^{-1}(\varphi_{j+1}{}_*\PPP(1^{2g-2}))$. Further, any such cycle must push forward under the map forgetting any of the marked points to give a non-zero cycle, which must then be proportional to the rigid and extremal cycle $ [\varphi_{j+1}{}_*\PPP(1^{2g-2})]$ and we obtain that $V_l$ is supported on
\begin{equation*}
\bigcap_{m=1}^{2g-2-j} \pi_m^{-1}(\varphi_{j+1}{}_*\PPP(1^{2g-2}))=\varphi_{j}{}_*\PPP(1^{2g-2})
\end{equation*}
providing a contradiction and proving the theorem.

In \S\ref{meromorphic} we use the methods of the previous section to obtain infinitely many rigid and extremal higher codimension cycles from meromorphic strata of canonical divisors of specified signatures.
\begin{restatable}{thm}{Meromorphic}
\label{meroextremal} 
For $g\geq2$ the cycle $[\varphi_j{}_*\overline{\PPP}(d_1,d_3,d_3,1^{2g-3})]$ for $j=0,\dots,g-1$ is extremal and rigid in $\mbox{Eff}^{\hspace{0.1cm}g-j}(\overline{\mathcal{M}}_{g,2g-j})$, where $\varphi_j:\overline{\mathcal{M}}_{g,2g}\longrightarrow\overline{\mathcal{M}}_{g,2g-j}$ forgets the last $j$ points, with $d_1+d_2+d_3=1$, $\sum_{d_i<0}d_i\leq-2$ and some $d_i=1$ if $g=2$. 
\end{restatable}
This gives the following corollary on the structure of the effective cone.
\begin{restatable}{cor}{Cone}
\label{cor:Cone}
$\text{Eff}^k(\Mgnb)$ is not rational polyhedral for $g\geq 2$ and $k\leq \min(n-g,g)$.
\end{restatable}
The rigid and extremal divisors of~\cite{Mullane3} are used as a base case in the inductive proof that employs the inductive strategy of the proof of Theorem~\ref{thm:Principal}. One complication occurs in the case of $g=2$ and $3$ for the codimension two case. This method provides two candidates for the cycle $V_l$ discussed above. This complication is overcome by observing that if the second candidate appears in an effective decomposition with multiplicity $c_l>0$ then the cycle 
\begin{equation*}
[\varphi_{j}{}_*\overline{\PPP}(d_1,d_3,d_3,1^{2g-3})-c_lV_l]
\end{equation*}
is effective. Hence pushing this cycle forward under forgetful morphisms in the cases of interest results in an effective divisor which must have non-negative intersection with the covering curves introduced in \S\ref{movingcurves}. This provides the required contradiction.

In \S\ref{genusone} we examine the genus $g=1$ case. In this case the meromorphic strata of canonical divisors have codimension one and to produce rigid and extremal higher codimension cycles we intersect the pullbacks of strata under forgetful morphisms. Set $1\leq m\leq n-1$ and let $\underline{d}^j=(d_1^j,\dots,d_{n-m+1}^j)$ for $j=1,\dots,m$ be distinct non-zero integer partitions of zero. We define
\begin{equation*}
X(\underline{d}^1,\dots,\underline{d}^m):=\left\{[E,p_1,\dots,p_n]\in\mathcal{M}_{1,n}\hspace{0.15cm}\big| \hspace{0.15cm}[E,p_1,\dots,p_{n-m},p_{n-m+j}]\in \PPP(\underline{d}^j)   \right\}.
\end{equation*}
Under certain conditions, we can show irreducibility.
\begin{restatable}{prop}{Irred}
\label{irred}
$X(\underline{d}^1,\dots,\underline{d}^m)$ is irreducible if $\gcd(\underline{d}^1)=1$ and $d_{n-m+1}^j=1$ for $j=2,\dots,m$. 
\end{restatable}
By the inductive strategy of earlier sections we obtain.
\begin{restatable}{thm}{Genusone}
\label{thm:genusone}
Let $m\geq 1$ and $n\geq m+2$, then $\overline{X}(\underline{d}^1,\dots,\underline{d}^m)$ is rigid and extremal in $\mbox{Eff}^{\hspace{0.1cm}m}(\overline{\mathcal{M}}_{1,n})$, if $\gcd(\underline{d}^1)=1$ and $d_{n-m+1}^j=1$ for $j=2,\dots,m$.
\end{restatable}
This immediately gives the following corollary on the structure of the effective cones.
\begin{restatable}{cor}{Genusonecor}
\label{genusonecor}
The effective cone of codimension $k$ cycles in $\overline{\mathcal{M}}_{1,n}$ is not rational polyhedral for $k\leq n-2$.
\end{restatable}

\section{Preliminaries}

\subsection{Subvarieties from the strata of canonical divisors}
The subvarieties of interest in this paper are the \emph{strata of canonical divisors with signature $\kappa$} defined as
\begin{equation*}
\PPP(\kappa):=\{[C,p_1,...,p_n]\in \Mgn   \hspace{0.15cm}| \hspace{0.15cm}k_1p_1+...+k_np_n\sim K_C\}.
\end{equation*}
The codimension of $\PPP(\kappa)$ is $g-1$ for $\kappa$  holomorphic  (all $k_i\geq0$)  and $g$ for $\kappa$ meromorphic (some $k_i<0$). 

We obtain many interesting subvarieties of lower codimension by pushing forward under the forgetful morphisms forgetting marked points. Let $\varphi_S:\Mgnb\longrightarrow \overline{\mathcal{M}}_{g,n-|S|}$ for $S\subseteq\{1,...,n\}$ be the map that forgets the marked points indexed by $S$. For ease of notation we will let $\varphi_j$ denote the map that forgets the last $j$ points, that is, $\varphi_j=\varphi_S$ for $S=\{n-j+1,...,n\}$. Further, we will let $\pi_j$ denote the map that forgets only the $j$th point, that is, $\pi_j=\varphi_S$ for $S=\{j\}$.

Despite many remaining interesting questions, the codimension one case is well studied. We obtain a codimension one subvariety or divisor from $\overline{\PPP}(\kappa)$ in the moduli space of marked genus $g$ curves by forgetting marked points.
The divisor $D^n_\kappa$  in $\Mgnb$ for $\kappa=(k_1,...,k_{n+s})$ with $\sum k_i=2g-2$ is defined as 
\begin{equation*}
D^n_\kappa=\overline{\{[C,p_1,...,p_n]\in\Mgn\hspace{0.15cm}|\hspace{0.15cm} [C,p_1,...,p_{n+s}]\in\mathcal{M}_{g,n+s} \text{ with  }  \sum k_ip_i\sim K_C   \}},
\end{equation*}
where $s=g-2$ or $g-1$ for holomorphic and meromorphic signature $\kappa $ respectively. Hence $D^n_\kappa$ is proportional to $[\varphi_s{}_*\overline{\PPP}(\kappa)]$.

\subsection{Degeneration of differentials}
A stable pointed curve $[C,p_1,...,p_n]\in\Mgnb$ is contained in $\widetilde{\PPP}(\kappa)$, the moduli space of twisted canonical divisors of type $\kappa=(k_1,...,k_n)$ as defined by Farkas and Pandharipande~\cite{FP} if there exists a collection of (possibly meromorphic) divisors $D_j\sim K_{C_j}$ on each irreducible component $C_j$ of $C$ such that
\begin{enumerate}
\item
The support of $D_j$ contains the set of marked points and the nodes lying in $C_j$, moreover if $p_i\in C_j$ then $\ord_{p_i}(D_j)=k_i$.
\item
If $q$ is a node of $C$ and $q\in C_i\cap C_j$ then $\ord_q(D_i)+\ord_q(D_j)=-2$.
\item
If $q$ is a node of $C$ and $q\in C_i\cap C_j$ such that $\ord_q(D_i)=\ord_q(D_j)=-1$ then for any $q'\in C_i\cap C_j$ we have $\ord_{q'}(D_i)=\ord_{q'}(D_j)=-1$. We write $C_i\sim C_j$.
\item
If $q$ is a node of $C$ and $q\in C_i\cap C_j$ such that $\ord_q(D_i)>\ord_q(D_j)$ then for any $q'\in C_i\cap C_j$ we have $\ord_{q'}(D_i)>\ord_{q'}(D_j)$. We write $C_i\succ C_j$.
\item
There does not exist a directed loop $C_1\succeq C_2\succeq...\succeq C_k\succeq C_1$ unless all $\succeq$ are $\sim$.
\end{enumerate}  
Farkas and Pandharipande showed that in addition to the main component $\overline{\PPP}(\kappa)$ containing $\PPP(\kappa)$, this space contained extra components completely contained in the boundary of the moduli space.
Bainbridge, Chen, Gendron, Grushevsky and M\"oller~\cite{BCGGM} provided the condition that a twisted canonical divisor lies in the main component. Let $\Gamma$ be the dual graph of $C$. A twisted canonical divisor of type $\kappa$ is the limit of twisted canonical divisors on smooth curves if there exists a collection of meromorphic differentials $\omega_i$ on $C_i$ with $(\omega_i)=D_i$ that satisfy the following conditions
\begin{enumerate}
\item
If $q$ is a node of $C$ and $q\in C_i\cap C_j$ such that $\ord_q(D_i)=\ord_q(D_j)=-1$ then $\res_q(\omega_i)+\res_q(\omega_j)=0$.
\item
There exists a full order on the dual graph $\Gamma$, written as a level graph $\overline{\Gamma}$, agreeing with the order of $\sim$ and $\succ$, such that for any level $L$ and any connected component $Y$ of  $\overline{\Gamma}_{>L}$ that does not contain a prescribed pole we have
\begin{equation*}
\sum_{\begin{array}{cc}\text{level}(q)=L, \\q\in C_i\subset Y\end{array}}\res_{q}(\omega_i)=0
\end{equation*} 
\end{enumerate}
Part (b) is known as the \emph{global residue condition}.

\subsection{Rigid and extremal cycle classes}

For a projective variety $X$, let $N^k(X)$ denote the $\RR$-vector space of codimension-$k$ cycles modulo numerical equivalence. The cycles in $N^k(X)$  that can be written as a positive sum of effective cycles form a convex cone inside $N^k(X)$ known as the \emph{effective cone} of codimension-$k$ cycles denoted $\text{Eff}^k(X)$.

An effective codimension-$k$ cycle $Y$ is \emph{extremal} or spans an extremal ray in the effective cone if the class $Y$ cannot be written as a sum $m_1Y_1+m_2Y_2$ of effective $Y_i$ with $m_1,m_2>0$ unless $Y, Y_1$ and $Y_2$ are all proportional classes. An effective cycle $Y$ is \emph{rigid} if every cycle with class $mY$ is supported on the support of $Y$ for every positive integer $m$.

The codimension one case is special. A curve $B$ contained in an effective divisor $D$ is known as a \emph{covering curve} for $D$ if irreducible curves with numerical class equal to $B$ cover a Zariski  dense subset of $D$. Negative intersection by a covering curve is a well-known criterion for an irreducible effective divisor to be extremal and rigid.

\subsection{Rigid and extremal divisors}
In this section we collect the known results on rigid and extremal divisor classes relevant to our later arguments. 

The closure of the locus of $g$ points on general genus $g$ curves that sit in a hyperplane section of the canonical embedding form a divisor in $\overline{\mathcal{M}}_{g,g}$. The class of this divisor was first calculated by Logan~\cite{Logan} who used it to investigate the Kodaira dimension of $\Mgnb$. From our perspective this divisor is 
\begin{equation*}
[D^g_{1^{2g-2}}]=\frac{1}{(g-2)!}[\varphi_{g-2}{}_*\overline{\PPP}(1^{2g-2})].
\end{equation*}
Kontsevich and Zorich~\cite{KontsevichZorich} showed ${\PPP}(1^{2g-2})$ to be irreducible and hence $\varphi_j{}_*\overline{\PPP}(1^{2g-2})$ is irreducible for $j=0,...,g-2$. In the divisorial case, or the case $j=g-2$,  Farkas and Verra~\cite{FarkasVerraU} further showed that this divisor is rigid and extremal through the construction of a covering curve with negative intersection. 

\begin{prop}[\cite{FarkasVerraU}]\label{FV}
$D^g_{1^{2g-2}}$ is a rigid and extremal divisor in $\overline{\mathcal{M}}_{g,g}$ for all $g\geq 2$.
\end{prop}

On an elliptic curve $E$ the structure sheaf and canonical bundle coincide. Through the use of covering curves with negative intersection, Chen and Coskun~\cite{ChenCoskun} showed that the condition that points on an elliptic curve satisfy certain equations under the group law formed rigid and extremal divisors in $\overline{\mathcal{M}}_{1,n}$ for $n\geq 3$. From our perspective we state these results in the following proposition.

\begin{prop}[\cite{ChenCoskun}]\label{ChenCoskun}
The divisors $D^{n}_{d_1,\dots,d_n}$ are rigid and extremal in $\overline{\mathcal{M}}_{1,n}$ for
$\gcd(d_1,\dots,d_n)=1$ and $n\geq 3$.
\end{prop}

In~\cite{Mullane3}, the author also used covering curves with negative intersection to exhibit infinitely many rigid and extremal divisors in $\Mgnb$ for $g\geq2$ and $n= g+1$. Further, these divisors pullback under the forgetful morphisms $\Mgnb\longrightarrow \overline{\mathcal{M}}_{g,g+1}$ to give infinitely many rigid and extremal divisors in $\Mgnb$ for $g\geq2$ and $n\geq g+1$. 

\begin{prop}[\cite{Mullane3}]\label{Mullane}
The divisors $D^{g+1}_{d_1,d_2,d_3,1^{2g-3}}$ are rigid and extremal in $\overline{\mathcal{M}}_{g,g+1}$ for
$d_1+d_2+d_3=1$, $\sum_{d_i<0}d_i\leq-2$, for all $g\geq 2$.
\end{prop}

Similarly, these divisors are simply pushforwards of strata of canonical divisors with meromorphic signatures
\begin{equation*}
[D^{g+1}_{d_1,d_2,d_3,1^{2g-3}}]=\frac{1}{(g-1)!}[\varphi_{g-1}{}_*\overline{\PPP}(d_1,d_2,d_3,1^{2g-3})].
\end{equation*}
Boissy~\cite{Boissy} showed ${\PPP}(d_1,d_2,d_3,1^{2g-3})$ to be irreducible and hence $\varphi_j{}_*\overline{\PPP}(d_1,d_2,d_3,1^{2g-3})$ is irreducible for $j=0,...,g-1$.

%

\subsection{Enumerative geometry on a general curve}\label{CxC}
In this section we present results on finite maps that will be used in enumerative calculations in later sections.

For a general genus $g=2$ curve $C$ and non-zero integers $d_1,d_2$ consider the map
\begin{eqnarray*}
\begin{array}{cccccc}
f_{d_1,d_2}:&C\times C&\longrightarrow &\Pic^{d_1+d_2}(C)\\
&(q_1,q_2)&\longmapsto&\OO_C(d_1q_1+d_2q_2).
\end{array}
\end{eqnarray*}

\begin{prop}\label{finite}
For $d_1\ne\pm d_2$ the map $f_{d_1,d_2}$ is finite with degree $2d_1^2d_2^2$. Further, $f_{d_1,d_2}$ has simple ramification along the diagonal $\Delta$ and the locus $I$ of points $(q_1,q_2)$ that are conjugate under the unique hyperelliptic involution of $C$. The intersection $\Delta\cap I$ consists of the six Weierstrass points of $C$ and the ramification order at these points is $2$.

For $d_1=d_2$ the map $f_{d_1,d_2}$ is generically finite with degree $2d_1^2d_2^2$. Further, $f_{d_1,d_2}$ has simple ramification along $\Delta$ and contracts $I$.

For $d_1=-d_2$ the map $f_{d_1,d_2}$ is generically finite with degree $2d_1^2d_2^2$. Further, $f_{d_1,d_2}$ has simple ramification along $I$ and contracts $\Delta$.
\end{prop}

\begin{proof}
This generalises~\cite{ChenTarasca}. Fix $d_1,d_2$ and let $f=f_{d_1,d_2}$. Take a general point $e\in C$ and consider the isomorphism
\begin{eqnarray*}
\begin{array}{cccccc}
H:\Pic^{d_1+d_2}(C)&\longrightarrow &J(C)\\
L&\longmapsto&L\otimes\OO_C(-(d_1+d_2)e).
\end{array}
\end{eqnarray*}
Now let $F=H\circ f$. Then we have $\deg F=\deg f$ and
\begin{equation*}
F(q_1,q_2)=\OO_C\biggl(d_1(q_1-e)+d_2(q_2-e)\biggr).
\end{equation*}
Let $\Theta$ be the fundamental class of the theta divisor in $J(C)$. By \cite{ACGH} \S1.5 we have
\begin{equation*}
\deg \Theta^2=g!=2
\end{equation*}
and the locus of $\OO_C(k(x-e))$ for varying $x\in C$ has class $k^2\Theta$ in $J(C)$. Hence
\begin{eqnarray*}
\deg F&=&\deg F_*F^*([\OO_C])\\
&=&\deg \left(d_1^2d_2^2\Theta^2\right)\\
&=&2d_1^2d_2^2.
\end{eqnarray*}
Now consider the branch and exceptional locus of $F$. This is the genus $g=2$ case of the general genus case dealt with in [\cite{Mullane1}, \S2.6]. First we look locally analytically at $F$ around the points of interest. If $f_0d\omega,f_{1}d\omega$ is a basis for $H^0(C,K_C)$, then locally analytically the map becomes
\begin{eqnarray*}
(q_1,q_2)&\longmapsto&\biggl(\int_e^{q_1}d_1f_0d\omega+\int_e^{q_2}d_2f_0d\omega,\int_e^{q_1}d_1f_{1}d\omega+\int_e^{q_2}d_2f_{1}d\omega\biggr)
\end{eqnarray*}
modulo $H_1(C,K_C)$. The map on tangent spaces at any fixed point $(q_1,q_2)\in C\times C$ is the Jacobian of $F$ at the point, which is
\begin{equation*}
DF_{(q_1,q_2)}=\begin{pmatrix}f_0(q_1)&f_0(q_2)\\
f_1(q_1)&f_1(q_2) \end{pmatrix}\begin{pmatrix}d_1&0\\
0&d_2 \end{pmatrix}.
\end{equation*}
Ramification or contraction in the map $F$ occurs when the map on tangent spaces is not injective which takes place at the points where $\rk(DF_{(q_1,q_2)})<2$. The ramification index at a point $(q_1,q_2)$ will be equal to the vanishing order of the determinant of $DF_{(q_1,q_2)}$ at the point.  

This can be written locally analytically by the basis $d\omega$ and $\omega(\omega-\alpha)d\omega$ where $0$ is conjugate to $\alpha$. In local coordinates with $(q_1,q_2)=(s,t)$ we have
\begin{equation*}
DF_{(q_1,q_2)}=\begin{pmatrix}1&1\\
s(s-\alpha)&t(t-\alpha) \end{pmatrix}\begin{pmatrix}d_1&0\\
0&d_2 \end{pmatrix}
\end{equation*}
and
\begin{equation*}
\det(DF_{(q_1,q_2)})=d_1d_2(s-t)(s+t-\alpha).
\end{equation*}
The loci $s-t=0$ and $s+t-\alpha=0$ are $\Delta$ and $I$ respectively and intersect at the Weierstrass point $s=t=\alpha/2$. When these irreducible loci are contracted is clear from examining their image in $\Pic^{d_1+d_2}(C)$.
\end{proof}

\subsection{Moving curves}\label{movingcurves}
In addition to defining subvarieties and divisors, the strata of canonical divisors also yield interesting curves in $\Mgnb$. Taking a fibration of $\overline{\PPP}(\kappa)$ for a meromorphic signature $\kappa$ with $|\kappa|=m\geq g+1$ and $1\leq n\leq m$ we obtain the curve
\begin{equation*}
B^n_{\kappa}:=\overline{\bigl\{ [C,p_1,...,p_n]\in\mathcal{M}_{g,n} \hspace{0.15cm}\big| \hspace{0.15cm} \text{fixed general $[C,p_{g+2},...,p_m]\in\mathcal{M}_{g,m-g-1}$ and } \sum_{i=1}^m k_ip_i \sim K_C   \bigr\}   }.
\end{equation*}
That is, by fixing a meromorphic signature $\kappa$, a general genus $g$ curve $C$ and all but exactly $g+1$ points we obtain $B^n_{\kappa}$, a one dimensional subvariety of $\overline{\PPP}(\kappa)$. For $m=|\kappa|\geq n+g$ these curves provide moving curves in $\Mgnb$, which are curves that have non-negative intersection with all effective divisors.
\begin{prop}\label{covering}
For $g\geq 2$, a meromorphic signature $\kappa$ with $|\kappa|=g+n-1$ and $n\geq g+1$, $B^n_{\kappa,1,-1}$ is a moving curve in $\overline{\mathcal{M}}_{g,n}$. Further  
\begin{equation*}
B^n_{\kappa,1,-1}\cdot D^n_\kappa=0.
\end{equation*}
Hence all non-negative sums of the irreducible components of the divisor $D^n_\kappa$ lie on the boundary of the closure of the effective cone known as the pseudo-effective cone.
\end{prop}
\begin{proof}
~\cite[Theorem 1.1]{Mullane3} 
\end{proof}

In \S\ref{meromorphic} we will also require the intersection of $B^n_{\kappa,1,-1}$ with certain boundary divisors for specific $n$ and $\kappa$. We present the required intersection numbers in the following propositions.
\begin{prop}\label{intg2}
For $\kappa=(-h,1^2,h)$ with $h\geq2$, in $\overline{\mathcal{M}}_{2,3}$,
\begin{equation*}
B^3_{\kappa,1,-1}\cdot\delta_{0:\{2,3\}}=8h^2
\end{equation*}
\end{prop}
\begin{proof}
To find $B^3_{\kappa,1,-1}\cdot\delta_{0:\{2,3\}}$ we need to enumerate the limits of differentials of this signature with $p_2$ and $p_3$ sitting together on a rational tail. Hence we require the points $p_1$ and $p_2$ such that
\begin{equation*}
-hp_1+2p_2+hq_1+q_2-q_3\sim K_C
\end{equation*}
with $p_1\ne p_2$, $p_i\ne q_j$ and any limits that may occur with these points colliding that will satisfy the global residue condition.

To enumerate such points we consider the map
\begin{eqnarray*}
\begin{array}{cccccc}
f_{h,-2}:&C\times C&\longrightarrow &\Pic^{h-2}(C)\\
&(p_1,p_2)&\longmapsto&\OO_C(hp_1-2p_2).
\end{array}
\end{eqnarray*}
introduced in \S\ref{CxC}. Analysing the fibre of this map above $hq_1-q_2+q_3-K_C\in \Pic^{h-2}(C)$ will provide us with the solutions of interest. By Proposition~\ref{finite} for $h\geq3$ this map is finite of degree $8h^2$, simply ramified along the diagonal $\Delta$ and the locus of pairs of points that are conjugate under the hyperelliptic involution denoted $I$. For $h=2$ this map is generically finite of degree $8h^2=32$, contracts $\Delta$ and is simply ramified along $I$. 

For a general choice of $q_i$ the fibre will contain no solutions where $p_1$ and $p_2$ coincide with each other or any of the $q_i$. Hence we have found all solutions.
\end{proof}

\begin{prop}\label{intg3}
For $\kappa=(d_2,d_3,1^3,d_1)$ with $d_i\in\ZZ\setminus\{0\}$, $d_1\geq 2$ and $d_1+d_2+d_3=1$, in $\overline{\mathcal{M}}_{3,4}$,
\begin{equation*}
B^4_{\kappa,1,-1}\cdot\delta_{0:\{3,4\}}=24d_2^2d_3^2
\end{equation*}
\end{prop}
\begin{proof}
To find $B^4_{\kappa,1,-1}\cdot\delta_{0:\{3,4\}}$ we need to enumerate the limits of differentials of this signature with $p_3$ and $p_4$ sitting together on a rational tail. Hence we require the points $p_1,p_2,p_3$ such that for fixed general $q_i$,
\begin{equation*}
d_2p_1+d_3p_2+2p_3+q_1+d_1q_2+q_3-q_4\sim K_C
\end{equation*}
with $p_i\ne p_j$ for $i\ne j$ and $p_i\ne q_j$ and any limits that may occur with these points colliding that will satisfy the global residue condition.

Consider the map
\begin{eqnarray*}
\begin{array}{cccccc}
f:C^{3}&\longrightarrow &\Pic^{3-d_1}(C)\\
(p_1,p_2,p_3)&\longmapsto&O_C(d_2p_1+d_3p_2+2p_3).
\end{array}
\end{eqnarray*}
The fibre of this map above $K_C(-q_1-d_1q_2-q_3+q_4)\in\Pic^{3-d_1}(C)$ will give us the solutions of interest. Take a general point $e\in C$ and consider the isomorphism
\begin{eqnarray*}
\begin{array}{cccccc}
h:\Pic^{3-d_1}(C)&\longrightarrow &J(C)\\
L&\longmapsto&L\otimes\OO_C(-de).
\end{array}
\end{eqnarray*}
Now let $F=h\circ f$, then $\deg F=\deg f$. Observe
\begin{equation*}
F(p_1,p_2,p_3)=\OO_C(d_2(p_1-e)+d_3(p_2-e)+2(p_3-e)).
\end{equation*}
Let $\Theta$ be the fundamental class of the theta divisor in $J(C)$. By \cite{ACGH} \S1.5 we have
\begin{equation*}
\deg \Theta^g=g!=6
\end{equation*}
and the dual of the locus of $\OO_C(k(x-e))$ for varying $x\in C$ has class $k^2\Theta$ in $J(C)$. Hence
\begin{eqnarray*}
\deg F&=&\deg F_*F^*([\OO_C])\\
&=&\deg \left(d_2^2\Theta \cdot d_3^2\Theta\cdot 2^2\Theta   \right)\\
&=&24d_2^2d_3^2.
\end{eqnarray*}
As we have chosen the $q_i$ general, the general fibre will contain no points where the $p_i$ coincide with each other or with the $q_i$. Hence we have found all solutions.
\end{proof}

\section{Extremal cycles supported in the boundary}\label{exbdry}
In this section we investigate higher codimension effective cycles supported in the boundary of $\Mgnb$ to show $\text{Eff}^2(\Mgnb)$ is not finite polyhedral for $g\geq3$ and $n\geq g-1$. We follow the methods presented in~\cite{ChenCoskunH} using the infinitely many extremal effective divisors presented in~\cite{Mullane3}. 

Consider the gluing morphism 
\begin{equation*}
\alpha:\widehat{\Delta}_{1:\emptyset}=\overline{\mathcal{M}}_{g-1,n+1}\times \overline{\mathcal{M}}_{1,1}\longrightarrow \Delta_{1:\emptyset}\subset \Mgnb
\end{equation*}
for $g\geq 3$, which glues a $[C,p_1,\dots,p_{n+1}]\in\overline{\mathcal{M}}_{g-1,n+1}$ to $[E,q]\in\overline{\mathcal{M}}_{1,1}$ by identifying $p_{n+1}$ with $q$ to form a node.

The following proposition is presented in~\cite{ChenCoskunH}.

\begin{prop}\label{prod}
Let $X$ and $Y$ be projective varieties such that numerical equivalence and rational equivalence are the same for codimension $k$ cycles in $X$, $Y$ and $X\times Y$ respectively, with $\RR$-coefficients. Suppose $Z$ is an extremal effective cycle of codimension $k$ in $X$. Then $Z\times Y$ is an extremal effective cycle of codimension $k$ in $X\times Y$. 
\end{prop}
\begin{proof}
\cite[Corollary 2.4]{ChenCoskunH}
\end{proof}
Hence pulling back the infinitely many extremal divisors on $\overline{\mathcal{M}}_{g-1,n+1}$ for $g-1\geq 2$ and $n+1\geq (g-1)+1$ presented in Proposition~\ref{Mullane} provides infinitely many extremal effective divisors in $\widehat{\Delta}_{1:\emptyset}$ for $g\geq 3$ and $n\geq g-1$. To show these cycles pushforward to provide extremal codimension-two cycles in $\Mgnb$ we will require more machinery.  

For a morphism $f:X\longrightarrow Y$ between two complete varieties one associates an index to any $Z$, a subvariety of $X$,
\begin{equation*}
e_f(Z)=\dim(Z)-\dim(f(Z)).
\end{equation*}

\begin{prop}\label{higherextremal}
Let $\alpha:Y\longrightarrow X$ be a morphism between two projective varieties. Assume that $A_k(Y)\longrightarrow N_k(Y)$ is an isomorphism and that the composite $\alpha_*:A_k(Y)\longrightarrow A_k(X)\longrightarrow N_k(X)$ is injective. Moreover, assume that $f: X\longrightarrow W$ is a morphism to a projective variety $W$ whose exceptional locus is contained in $\alpha(Y)$. If a $k$-dimensional subvariety $Z\subset Y$ is an extremal cycle in $\text{Eff}_k(Y)$ and if $e_f(\alpha(Z))>0$, then $\alpha(Z)$ is also extremal in $\text{Eff}_k(X)$.
\end{prop}
\begin{proof}
\cite[Proposition 2.5]{ChenCoskunH}
\end{proof}
We apply this proposition to the situation $Y=\widehat{\Delta}_{1:\emptyset}$ and $X=\Mgnb$ where $f$ is the morphism
\begin{equation*}
ps:\Mgnb\longrightarrow \Mgnb^{ps}
\end{equation*}
that contracts unmarked elliptic tails to cusps. Indeed, the exceptional locus of $ps$ is $\Delta_{1:\emptyset}$. It remains to show that $\alpha_*:N^1(\widehat{\Delta}_{1:\emptyset})\longrightarrow N^2(\Mgnb)$ is injective. To this end, we introduce a number of test surfaces in $\Mgnb$. 

Consider the following test surfaces
\begin{itemize}
\item
$S^a$: Fix a general smooth curve $[C,q_1,\dots,q_{n}]\in\mathcal{M}_{g-1,n}$. Form the surface by attaching a general pencil of plane cubics at a base point to $q_{1}$ to form a node. Label $q_2,\dots, q_{k}$ as $p_2,\dots, p_k$ and allow the point $p_1$ to vary in the pencil. See Figure 1.
\begin{center}
\vspace{0.5cm}
\begin{overpic}[width=0.30\textwidth]{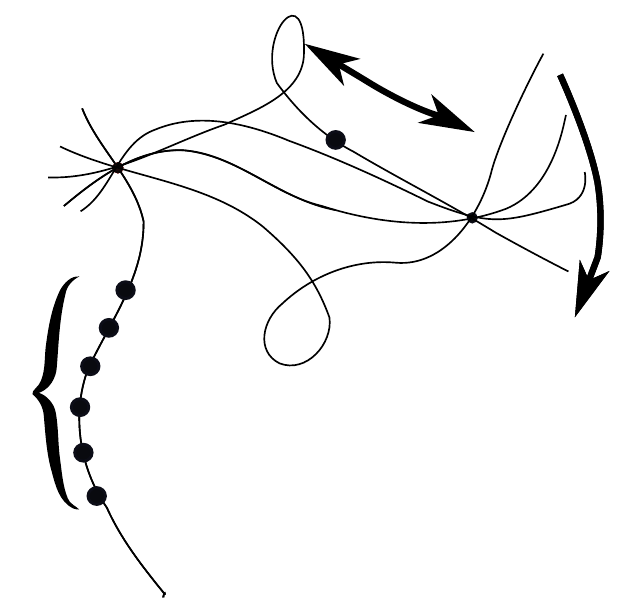}
\put (30, 3){$C$}
\put (58, 76){$p_1$}
\put (102, 64){Pencil of plane cubics}
\put (-28, 33){$p_2,\dots,p_n$}
\put (20,-10){\footnotesize{Figure 1: Test surface $S^a$.}}
\end{overpic}
\end{center}
\vspace{01cm}

\item
$S^b_{0:k}$: Fix $1\leq k\leq n$. Fix general smooth curves $[X,q_1,\dots,q_{k+1}]\in \mathcal{M}_{0,k+1}$, 
$[C,q'_1,\dots, q'_{n-k}]\in\mathcal{M}_{g-1,n-k}$ 
and $[E,q]\in\mathcal{M}_{1,1}$. Form the surface by attaching $q_{1}$ to $q$ and $q_{k+1}$ to a point $q'$ that varies freely in $C$, to form nodes. Label $q_2,\dots, q_{k}$ as $p_2,\dots, p_k$ and $q'_1,\dots,q'_{n-k}$ as $p_{k+1},\dots,p_{n}$. Allow the point $p_1$ to vary in  $E$. See Figure 2.
\begin{center}
\vspace{0.5cm}
\begin{overpic}[width=0.3\textwidth]{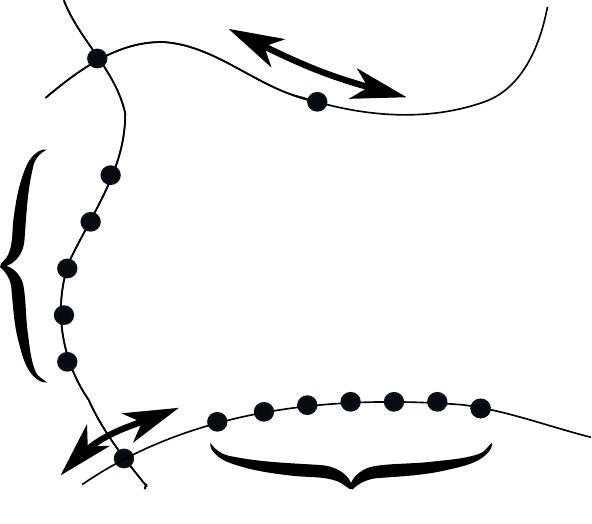}
\put (20, 44){$X$}
\put (95, 18){$C$}
\put (52, 62){$p_1$}
\put (82, 64){$E$}
\put (-38, 44){$p_2,\dots,p_k$ }
\put (45, 0){$p_{k+1},\dots,p_n$}
\put (20,-15){\footnotesize{Figure 2: Test surface $S^b_{0:k}$.}}
\end{overpic}
\end{center}
\vspace{1cm}
\item
$S^c$: Fix general smooth curves 
$[C,q_1,\dots, q_{n-1}]\in\mathcal{M}_{g-1,n-1}$ 
and $[E,q]\in\mathcal{M}_{1,1}$. Form the surface by attaching $q_{n-1}$ to $q$ to form a node and labelling $q_1,\dots,q_{n-2}$ as $p_3,\dots,p_n$. Allow $p_1$ and $p_2$ to vary freely in $E$ and $C$ respectively. See Figure 3.
\begin{center}
\vspace{0.5cm}
\begin{overpic}[width=0.3\textwidth]{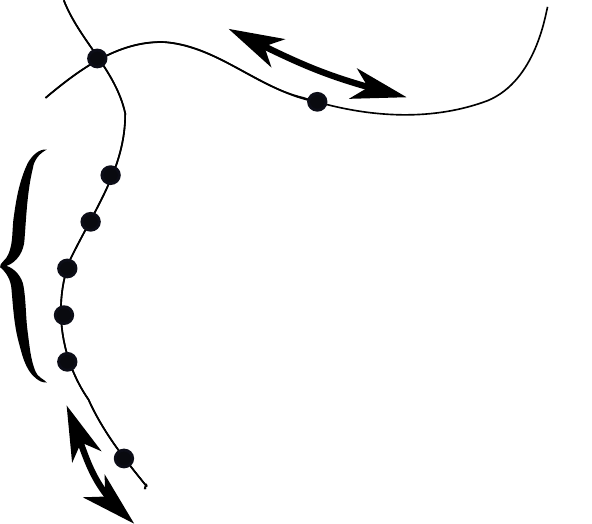}
\put (20, 44){$C$}
\put (52, 62){$p_1$}
\put (28, 11){$p_2$}
\put (82, 64){$E$}
\put (-38, 44){$p_3,\dots,p_n$ }
\put (20,-10){\footnotesize{Figure 3: Test surface $S^c$.}}
\end{overpic}
\end{center}
\vspace{0.5cm}
\item
$S^d_{0:k}$: Fix $3\leq k\leq n$. Fix general smooth curves $[X,q_1,\dots,q_{k}]\in\mathcal{M}_{0,k}$, $[C,q'_1,\dots, q'_{n-k+1}]\in\mathcal{M}_{g-1,n-k+1}$ and $[E,q]\in\mathcal{M}_{1,1}$. Form the surface by attaching the point $q'_{n-k+1}$ to a point $x$ that varies in $X$ to form a node. Label the points $q_2,\dots,q_{k}$ as $p_2,\dots,p_k$ and the points $q'_1,\dots,q'_{n-k}$ as $p_{k+1},\dots,p_n$. Attach the point $q$ to $q_1$ to form a node and allow $p_1$ to vary freely in $E$. See Figure 4.
\begin{center}
\begin{overpic}[width=0.3\textwidth]{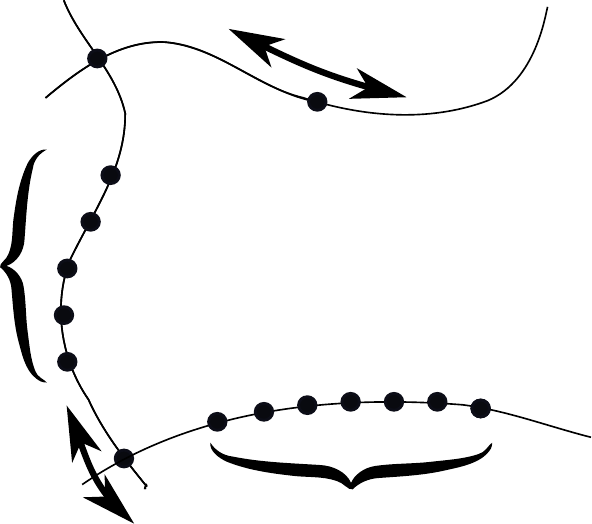}
\put (20, 44){$X$}
\put (95, 18){$C$}
\put (52, 62){$p_1$}
\put (82, 64){$E$}
\put (-38, 44){$p_2,\dots,p_k$ }
\put (45, 0){$p_{k+1},\dots,p_n$}
\put (20,-15){\footnotesize{Figure 4: Test surface $S^d_{0:k}$.}}
\end{overpic}
\end{center}
\vspace{1cm}

\item
$S^e_{h:k}$:  Fix $0\leq h\leq g-1$ and $1\leq k\leq n$ and further require if $h=0$ then $k\geq 3$, if $h=1$ then $k\geq 2$ and if $h=g-1$ then $k\leq n-1$. Fix general smooth curves $[C,q_1,\dots,q_{n-k+1},q]\in\mathcal{M}_{g-h-1,n-k+2}$ and $[C',q'_1,\dots,q'_{k}]\in\mathcal{M}_{h,k}$. Form the surface by identifying the point $q_{n-k+1}$ with the point $q'_{1}$, and the point  $q$ with a point $p$ that varies in $C$ to form nodes. Label $q_1,\dots,q_{n-k}$ as $p_{k+1},\dots,p_{n}$ and label $q'_2,\dots,q'_{k}$ as $p_2,\dots,p_k$. Let $p_1$ vary in $C'$. See Figure 5.
\begin{figure}[htbp]
\begin{center}
\begin{overpic}[width=0.3\textwidth]{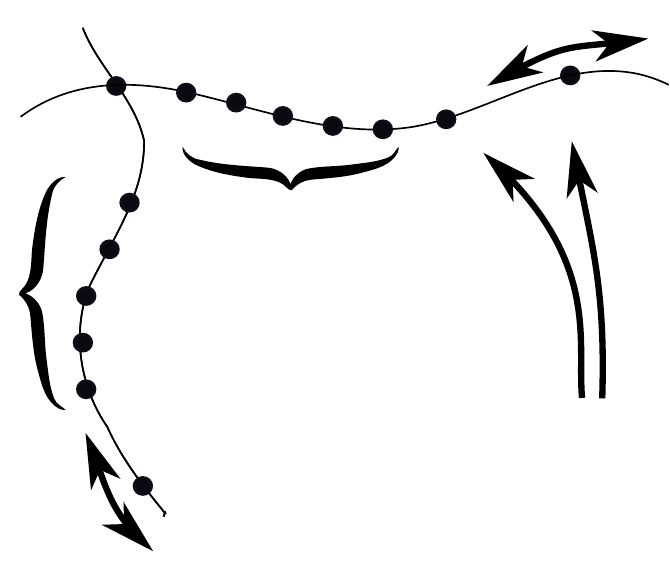}
\put (102, 70){$C$}
\put (24, 51){ $p_{k+1},\dots,p_n$ }
\put (82, 66){$p$}
\put (66, 61){$q$}
\put (65, 20){Identify $p$ and $q$}
\put (65, 10){to form a node}
\put (-33, 42){$p_2,\dots,p_k$}
\put (26, 14){$p_1$}
\put (25, 0){$C'$}
\put (20,-10){\footnotesize{Figure 5: Test surface $S^e_{h:k}$.}}
\end{overpic}
\end{center}
\vspace{0.5cm}
\end{figure}
\pagebreak
\item
$S^f_{h:k}$: Fix $0\leq h\leq g-2$ and $2\leq k\leq n$. Fix general smooth curves $[C,q_1,\dots,q_k]\in\mathcal{M}_{h,k}$, $[C',q'_1,\dots,q'_{n-k+2}]\in\mathcal{M}_{g-h-2,n-k+2}$ and $[E,q]\in\mathcal{M}_{1,1}$. Form the surface by attaching $q_{1}$ and $q'_{1}$ to distinct base points of a general pencil of plane cubics and $q'_{2}$ to $q$ to form nodes and labelling $q_2,\dots,q_k$ as $p_2,\dots,p_k$ and $q'_3,\dots,q'_{n-k+2}$ as $p_{k+1},\dots,p_n$. Allow $p_1$ to vary freely in $E$. See Figure 6.
\begin{figure}[htbp]
\begin{center}
\begin{overpic}[width=0.35\textwidth]{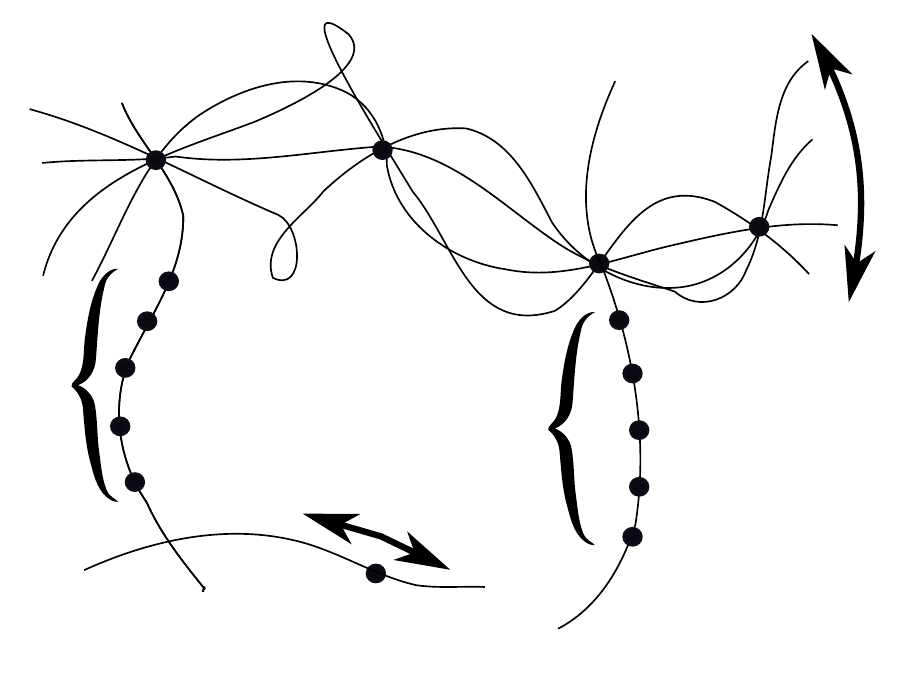}
\put (102, 56){Pencil of plane cubics}
\put (70, 5){$C$}
\put (38, 22){$p_1$}
\put (44, 2){$E$}
\put (24, 5){$C'$}
\put (-28, 31){$p_{k+1},\dots,p_n$ }
\put (32, 31){$p_{2},\dots,p_k$ }
\put (20,-10){\footnotesize{Figure 6: Test surface $S^f_{h:k}$.}}
\end{overpic}
\end{center}
\end{figure}
\vspace{0.5cm}
 \item
$S^g$: Fix general smooth curves $[C,q_1,\dots,q_{n+1}]\in\mathcal{M}_{g-3,n+1}$ and $[E,q]\in\mathcal{M}_{1,1}$. Form the surface by attaching a general pencil of plane cubics at a base point to $q_{n+1}$ and attaching $q_1$ to $q$ to form nodes.  Attach two other base points of the pencil together to form a node. Label points $q_2,\dots,q_{n}$ as $p_2,\dots,p_n$. Allow the point $p_1$ to vary freely in $E$. See Figure 7.
\begin{figure}[htbp]
\begin{center}
\begin{overpic}[width=0.35\textwidth]{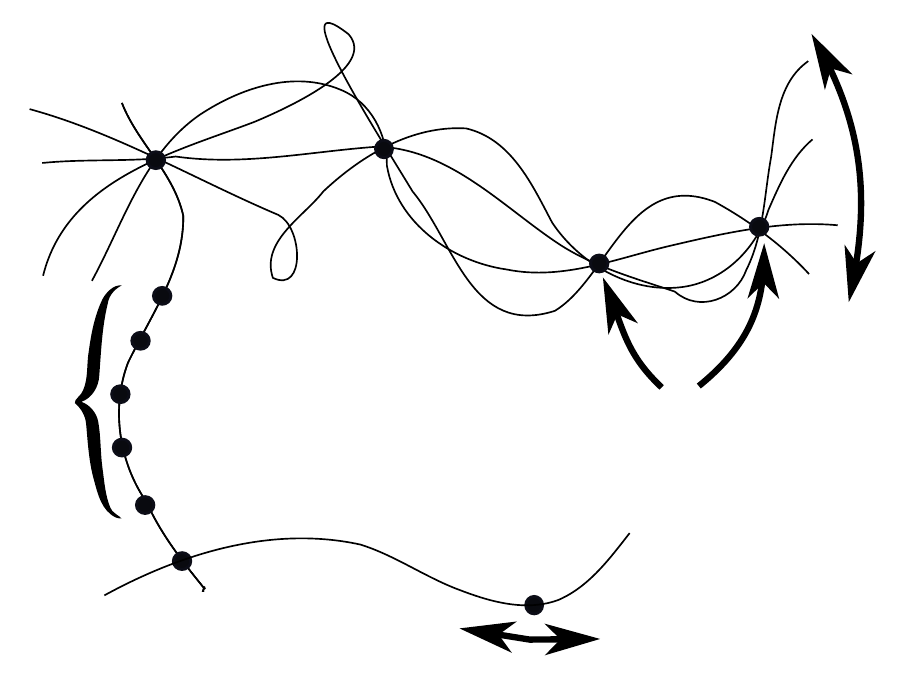}
\put (102, 56){Pencil of plane cubics}
\put (65, 26){Identify two base points to form a node}
\put (56, 13){$p_1$}
\put (70, 10){$E$}
\put (24, 5){$C$}
\put (-23, 30){$p_2,\dots,p_n$ }

\put (20,-10){\footnotesize{Figure 7: Test surface $S^g$.}}
\end{overpic}
\end{center}
\vspace{0.5cm}
\end{figure}
\pagebreak
\item
$S^h$: For $g\geq 5$ fix general smooth curves $[C,q_1,\dots,q_{n-1}]\in\mathcal{M}_{g-5,n-1}$, $[E',p,q']\in\mathcal{M}_{1,2}$,  and $[E,q]\in\mathcal{M}_{1,1}$. Form the surface by attaching $p,q_{1}$ and $q$ to distinct base points of a general pencil of plane quartics to form nodes and labelling $q'$ as $p_2$ and $q_2,\dots,q_{n-1}$ as $p_3,\dots,p_n$. Allow the point $p_1$ to move freely in $E$. See Figure 8. 
\begin{center}
\begin{overpic}[width=0.35\textwidth]{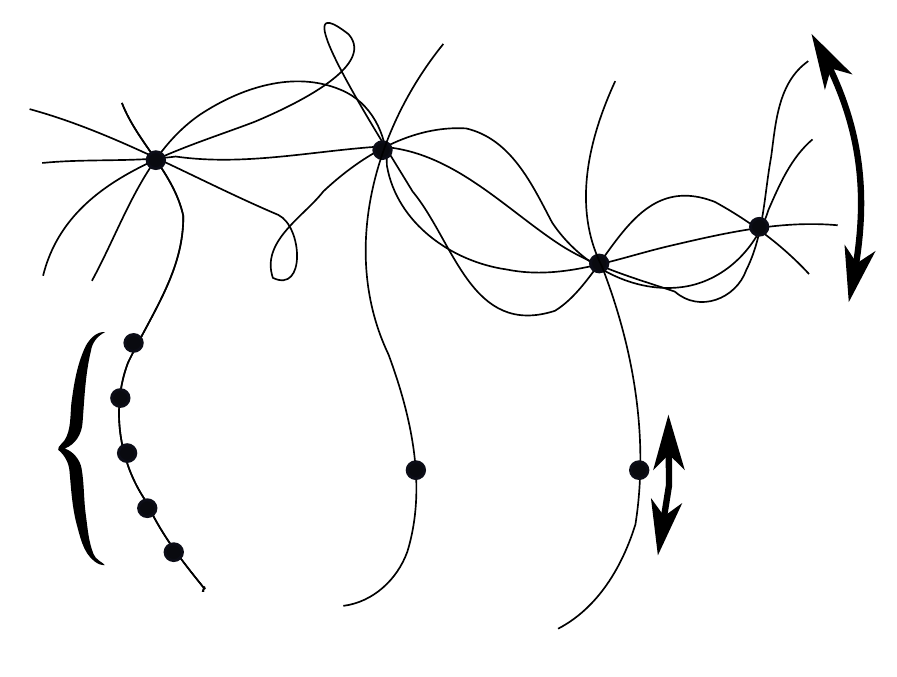}
\put (102, 56){Pencil of plane quartics}
\put (62, 22){$p_1$}
\put (70, 5){$E$}
\put (38, 22){$p_2$}
\put (44, 5){$E'$}
\put (24, 5){$C$}
\put (-25, 25){$p_3,\dots,p_n$ }
\put (20,-10){\footnotesize{Figure 8: Test surface $S^{h}$.}}
\end{overpic}
\end{center}
\vspace{0.5cm}
\end{itemize}

Define the cycle $\Gamma$ as the image under $\alpha$ of $\overline{\mathcal{M}}_{g-1,n+1}\times \delta_0$ and $\Gamma_{i:S},\Gamma_{K_i},\Gamma_\lambda, \Gamma_0$ as the image under $\alpha$ of the pullback of $\delta_{i:S}, K_i,\lambda$ and $\delta_0$ respectively, under the forgetful morphism 
\begin{equation*}
\widehat{\Delta}_{1:\emptyset}=\overline{\mathcal{M}}_{g-1,n+1}\times \overline{\mathcal{M}}_{1,1}\longrightarrow \overline{\mathcal{M}}_{g-1,n+1},
\end{equation*}
where $K_i$ is defined as $\pi_i^*\psi$ where $\pi_i:\overline{\mathcal{M}}_{g-1,n+1}\longrightarrow \overline{\mathcal{M}}_{g-1,1}$ is the projection to the $i$th point.

Since $\overline{\mathcal{M}}_{1,1}\cong \PP^1$ we have $A^1(\widehat{\Delta}_{1:\emptyset})\cong N^1(\widehat{\Delta}_{1:\emptyset})$ and is generated by the classes of the cycles $\Gamma,\Gamma_\lambda,\Gamma_0,\Gamma_{i:S},\Gamma_{K_i}$  for $g\geq 4$, when $g=3$ we omit $\Gamma_0$ to account for the relation between $\lambda$ and the boundary classes in $\overline{\mathcal{M}}_{2,n+1}$.

\begin{prop}\label{SurfacesIntersection}
Let $\gamma,\gamma_\lambda,\gamma_0,\gamma_{i:S},\gamma_{K_i}$ denote the classes in $\Mgnb$ of the cycles $\Gamma,\Gamma_\lambda,\Gamma_0,\Gamma_{i:S},\Gamma_{K_i}$ respectively. The surfaces have the following intersection numbers.
\begin{eqnarray*}
&S^a\cdot \gamma=12,\hspace{1cm} S^a\cdot \gamma_{K_{n+1}}=1,&\\
&S^b_{0:k}\cdot \gamma_{0:\{1,\dots,k\}\cup\{j,n+1\}}=1\hspace{0.6cm}\text{for $j=k+1,\dots,n$},\hspace{1cm}S^b_{0:k}\cdot \gamma_{0:\{1,\dots,k\}\cup\{n+1\}}=2-2(g-1)-(n+1-k),&\\
&S^b_{0:k}\cdot \gamma_{K_j}=2g-2\hspace{0.6cm}\text{ for $j=1,\dots,k$ and $j=n+1,$}&\\
&S^c\cdot\gamma_{0:\{2,j\}}=1,\hspace{0.3cm}\text{for $j=3,\dots,n$,} \hspace{0.5cm}S^c\cdot\gamma_{0:\{1,n+1\}}=-1, \hspace{0.5cm}S^c\cdot\gamma_{0:\{1,2,n+1\}}=1, \hspace{0.5cm}S^c\cdot\gamma_{K_2}=2(g-1)-2, &\\ 
&S^d_{0:k}\cdot\gamma_{g-1:\{k+1,\dots,n\}\cup\{j\}}=1\text{ for $j=2,\dots,k$},\hspace{0.5cm} S^d_{0:k}\cdot\gamma_{g-1:\{k+1,\dots,n\}}=2-k, \hspace{0.5cm}S^d_{0:k}\cdot\gamma_{0:\{1,n+1\}}=-1, &\\
&S^d_{0:k}\cdot\gamma_{0:\{2,\dots,k\}}=1, &\\
&S^e_{h:k}\cdot \gamma_{h:\{1,\dots,k\}}=-1,\hspace{1cm} S^e_{h:k}\cdot \gamma_{h:\{2,\dots,k\}}=1,\hspace{1cm} S^e_{h:k}\cdot \gamma_{0:\{1,j\}}=1\hspace{0.3cm}\text{for $j=2,\dots,k$},&\\
&S^e_{h:k}\cdot\gamma_{K_1}=\begin{cases} 2(g-1)-2 &\text{ for $h=g-1$},\\
 2h-1 &\text{ for $h=1,\dots,g-2$},\\ 
0&\text {for $h=0$} \end{cases}&\\
&S^f_{h:k}\cdot\gamma_{h:\{2,\dots,k\}}=-1,\hspace{0.5cm} S^f_{h:k}\cdot\gamma_{h+1:\{2,\dots,k\}}=-1, \hspace{0.5cm}S^f_{h:k}\cdot\gamma_\lambda=1, \hspace{0.5cm}S^f_{h:k}\cdot\gamma_0=12, &\\
&S^f_{h:k}\cdot\gamma_{K_j}=\begin{cases} 1 &\text{ for $h=0$, $j=2,\dots,k$},\\
 1 &\text{ for $h=g-2$, $j= 1,k+1,\dots,n+1$},\\ 
0&\text {otherwise,} \end{cases}&\\
\end{eqnarray*}
\begin{eqnarray*}
&S^g\cdot\gamma_\lambda=1, \hspace{0.3cm}S^g\cdot\gamma_0=10, \hspace{0.3cm}S^g\cdot\gamma_{2:\emptyset}=-1,  \hspace{0.3cm}S^g\cdot\gamma_{K_j}=\begin{cases}1&\text{ for $g=3$ for all $j$},\\ 0 &\text{ otherwise},  \end{cases}&\\
&S^h\cdot\gamma_{1:\{2\}}=-1, \hspace{1cm}S^h\cdot\gamma_{4:\{1,2,n+1\}}=-1, \hspace{1cm}S^h\cdot\gamma_{0:\{1,n+1\}}=-1,&\\
&S^h\cdot\gamma_{K_j}=\begin{cases} 1 &\text{ for $j=1,n+1$ for all $g\geq 5$},\\
 1 &\text{ for $j=3,\dots,n$ when $g=5$},\\ 
0&\text {otherwise}, \end{cases}&\\
& \hspace{1cm}S^h\cdot\gamma_{\lambda}=3, \hspace{1cm}S^h\cdot\gamma_{0}=27.&\\
\end{eqnarray*}
All other intersections are zero.
\end{prop}

\begin{proof}
The intersection of each surface with $\Delta_{1,\emptyset}$ is transverse providing a curve to intersect with $\Gamma,\Gamma_\lambda,\Gamma_0,\Gamma_{i:S},\Gamma_{K_j}$ inside $\Delta_{1,\emptyset}$. The intersection numbers are a simple exercise in intersection theory~\cite{HarrisMorrison}.  Here we provide the example of the intersection numbers for the surface $S^d_{0:k}$. First observe that the only non-zero intersections are with classes
\begin{equation*}
\gamma_{0:\{1,n+1\}}, 
 \hspace{0.5cm}\gamma_{0:\{2,\dots,k\}},  \hspace{0.5cm} \gamma_{g-1:\{k+1,\dots,n\}},\hspace{0.5cm} \gamma_{g-1:\{k+1,\dots,n\}\cup\{j\}}\text{ for $j=2,\dots,k$}.
\end{equation*}
Hence
\begin{eqnarray*}
&S^d_{0:k}\cdot \gamma_{0:\{1,n+1\}}=S^d_{0:k}\cdot \delta_{1:\emptyset}\cdot \delta_{1:\{1\}}, 
 \hspace{0.5cm}S^d_{0:k}\cdot \gamma_{0:\{2,\dots,k\}}=S^d_{0:k}\cdot \delta_{1:\emptyset}\cdot\delta_{0:\{2,\dots,k\}},& \\ &S^d_{0:k}\cdot \gamma_{g-1:\{k+1,\dots,n\}}=S^d_{0:k}\cdot \delta_{1:\emptyset}\cdot \delta_{g-1:\{k+1,\dots,n\}},&\\
 &S^d_{0:k}\cdot \gamma_{g-1:\{k+1,\dots,n\}\cup\{j\}}=S^d_{0:k}\cdot \delta_{1:\emptyset}\cdot\delta_{g-1:\{k+1,\dots,n\}\cup\{j\}} \text{ for $j=2,\dots,k$},&
\end{eqnarray*}
as $\gamma_{0:\{2,\dots,k\}}= \delta_{1:\emptyset}\cdot\delta_{0:\{2,\dots,k\}}$ and $\gamma_{0:\{1,n+1\}}$ forms the only component of $\delta_{1:\emptyset}\cdot \delta_{1:\{1\}}$ that has non-zero intersection with $S^d_{0:k}$ (similarly for $ \gamma_{g-1:\{k+1,\dots,n\}}$ and $\gamma_{g-1:\{k+1,\dots,n\}\cup\{j\}}$ for $j=2,\dots,k$).

Further, we denote by $B$ the curve $S^d_{0:k}\cdot \delta_{1:\emptyset}$ formed by fixing the point $p_1$ in the surface $S^d_{0:k}$ to sit on a rational bridge between $E$ and $X$. Hence 
\begin{equation*}
S^d_{0:k}\cdot \gamma_{0:\{2,\dots,k\}}=B\cdot\delta_{0:\{2,\dots,k\}}=1
\end{equation*}
and
\begin{equation*}
S^d_{0:k}\cdot \gamma_{g-1:\{k+1,\dots,n\}\cup\{j\}}=B\cdot\delta_{g-1:\{k+1,\dots,n\}\cup\{j\}} =1
\end{equation*}
for $j=2,\dots,k$.

Let $S$ be the surface obtained by blowing up $X\times X$ at the points $(q_i,q_i)$ for $i=1,\dots,k$ and let $\pi:S\longrightarrow X\times X$ be the blowdown. If we let $T=X\times\{q_1\}$ in $X\times X$ and $\widetilde{T}$ the proper transform in $S$ then
\begin{equation*}
B\cdot \delta_{0\{1,n+1\}}=\deg(\mathcal{N}_{\widetilde{T}/S}\otimes\mathcal{N}_{X\times\{q\}/X\times E})=\widetilde{T}^2=-1.
\end{equation*}
Further, if $\Delta$ denotes the diagonal in $X\times X$ and $\widetilde{\Delta}$ the proper transform in $S$ we have
\begin{equation*}
B\cdot\delta_{g-1:\{k+1,\dots,n\}}=\deg(\mathcal{N}_{\widetilde{\Delta}/S}\otimes \mathcal{N}_{X\times\{q'_{n-k+1}\}/X\times C})=\widetilde{\Delta}^2=2-k.
\end{equation*}
\end{proof}

\begin{prop}\label{injective}
$\alpha_*: N^1(\widehat{\Delta}_{1:\emptyset})\longrightarrow N^2(\Mgnb)$ is injective for $g\geq3$ and $n\geq g-1$.
\end{prop}
\begin{proof}
Consider a non-trivial relation on the classes  in $\Mgnb$
\begin{equation}\label{relation}
c_\gamma\gamma+c_\lambda\gamma_\lambda+c_0\gamma_0+\sum_{i,S} c_{i:S}\gamma_{i:S}+\sum_{j=1}^{n+1}c_{K_j}\gamma_{K_j}=0,
\end{equation}
where for $g=3$ we assume $c_0=0$ to account for the known relation between $\lambda$ and the boundary classes in $\overline{\mathcal{M}}_{2,n+1}$. By intersecting this equation with the surfaces presented in Proposition~\ref{SurfacesIntersection} we obtain relations on the coefficients. 

Averaging any non-trivial relation under action of $S_n$ permuting the marked points on $\Mgnb$ gives a (possibly trivial) relation such that $c_{K_i}=c_{K_j}$ for $i,j\ne n+1$ and $c_{i:S}=c_{i:T}$ for $|S|=|T|$ and $|S\cup \{n+1\}|=|T\cup \{n+1\}|$. We first show that any relation of this type is trivial. We assume these symmetries hold in equation~(\ref{relation}) and denote $c_K=c_{K_i}$ for $i\ne n+1$ and $c_{i:|S|}=c_{i:S}$ for $n+1\nin S$.

Consider $ps:\Mgnb\longrightarrow\Mgnb^{ps}$ that contracts unmarked elliptic tails. All cycles in~(\ref{relation}) except $\gamma$ are contracted by this morphism. Hence pushing down the relation we obtain
\begin{equation*}
c_\gamma=0.
\end{equation*}
Test surface $S^a$ then immediately implies 
\begin{equation*}
c_{K_{n+1}}=0.
\end{equation*}
Test surface $S^b_{0:k}$ for $k=n$ implies 
\begin{equation}\label{Sb0}
(2g-2)nc_K-(2g-2)c_{g-1:0}=0,
\end{equation}
and more generally for $k=1,\dots,n-1$,
\begin{equation}\label{Sbk}
(2g-2)kc_K-(2-2(g-1)-(n+1-k))c_{g-1:n-k} +(n-k)c_{g-1:n-k-1} =0.
\end{equation}
Test surface $S^c$ gives
\begin{equation}\label{Sc}
(2(g-1)-2)c_K+(n-2)c_{0:2}-c_{g-1:n-1}+c_{g-1:n-2}=0.
\end{equation}
Now consider $\pi_i:\Mgnb\longrightarrow \overline{\mathcal{M}}_{g,n-1}$ that forgets the $i$th marked point. Pushing forward we obtain 
\begin{equation*}
\pi_i{}_*\gamma_{0:\{i,j\}}=\delta_{1,\emptyset}\text{ for $j=1,\dots,\hat{i},\dots,n+1$}, \hspace{0.8cm}\pi_i{}_*\gamma_{K_i}=(2(g-1)-2)\delta_{1,\emptyset},
\end{equation*}
with all other cycles pushing forward to give zero. Hence pushing forward~(\ref{relation}) we obtain
\begin{equation}\label{push}
(2(g-1)-2)c_{K}+(n-1)c_{0:2}+c_{g-1:n-1}=0.
\end{equation} 
For $g=3$ and $n=2$ equations~(\ref{Sb0}),(\ref{Sbk}),(\ref{Sc}) and (\ref{push}) give independent relations, hence 
\begin{equation*}
c_{0:2}=c_{g-1:0}=c_{g-1:1}=c_K=0.
\end{equation*}
For $g\geq3$ and $n\geq \max(g-1,3)$ we have $S^d_{0:k}$ for $k=3$ yields
\begin{equation}\label{Sc0}
c_{0:2}-c_{g-1:n-3}+2c_{g-1:n-2}-c_{g-1:n-1}=0.
\end{equation}
In this case, equations~(\ref{Sbk}) with $k=1$ and $2$, (\ref{Sc}), (\ref{push}) and~(\ref{Sc0}) provide independent relations, hence
\begin{equation*}
c_{0:2}=c_{g-1:n-3}=c_{g-1:n-2}=c_{g-1:n-1}=c_K=0.
\end{equation*}
Hence by~(\ref{Sb0}) and (\ref{Sbk}),
\begin{equation}\label{Cgm1}
c_{g-1:s}=0
\end{equation}
for $s=0,\dots n-1$.

Test surface $S^e_{h:s}$ gives the relation between the remaining boundary coefficients
\begin{equation}\label{Seh}
c_{h:s}=c_{h:s-1}
\end{equation} 
where if $h=0$ then $3\leq s\leq n$, if $h=1$ then $2\leq s\leq n$, if $h=g-1$ then $1\leq s\leq n-1$ and if $h= 2,\dots,g-2$ then $1\leq s\leq n$. Hence as $c_{0:2}=0$, setting $h=0$ obtains
\begin{equation*}
c_{0:s}=0
\end{equation*}
for $s=2,\dots,n$. 

Test surface $S^f_{h:2}$ gives the relation
\begin{equation*}
-c_{1:1}+c_\lambda+12c_0=0
\end{equation*}
for $h=0$ and 
\begin{equation*}
-c_{h:1}-c_{h+1:1}+c_\lambda+12c_0=0
\end{equation*}
for $1\leq h\leq g-2$. For $g=3$ we omit the $c_0$ term. Comparing this equation for consecutive $h$ gives
\begin{equation*} 
c_{h:1}=c_{h+2:1}
\end{equation*}
for $0\leq h\leq g-3$.   Combined with equation~(\ref{Seh}) this gives
\begin{equation*}
c_{h:s}=0
\end{equation*} 
for all even $h$ with $0\leq s\leq n$. For $g$ even this extends by~(\ref{Cgm1}) to all $h$ and $s$ except $h=1$, $s=0$.

For $g$ odd we have
\begin{equation}\label{ij}
c_{i:s}=c_{j:t}=c_\lambda+12c_0
\end{equation}
for $i,j$ odd $0\leq s,t\leq n$ but $s,t\ne 0$ if $i,j=1$. 

Test surface $S^g$ gives the relation
\begin{equation}\label{quartic}
c_\lambda+10c_0=0.
\end{equation}
For $g=3$ we omit the $c_0$ term and hence $c_\lambda=0$ and~(\ref{ij}) implies $c_{1:s}=0$ for $1\leq s\leq n$. Hence the one remaining coefficient in (\ref{relation}) must be $c_{1:0}=0$ and the relation is trivial.

For $g\geq 4$ even,~(\ref{ij}) gives
\begin{equation*}
c_\lambda+12c_0=0.
\end{equation*} 
Equation~(\ref{quartic}) then implies $c_\lambda=c_0=0$ and the one remaining coefficient in (\ref{relation}) must be $c_{1:0}=0$, hence the relation is trivial for even $g\geq4$.

For $g\geq 5$ odd, test surface $S^h$ gives
\begin{equation*}\label{Sh}
-c_{1:1}+3c_\lambda+27c_0=0,
\end{equation*}
which with~(\ref{ij}) and (\ref{quartic}) implies 
\begin{equation*}
c_\lambda=c_0=c_{i:s}=0
\end{equation*}
for all $i,s$ except $i=1,s=0$ and the only remaining coefficient in (\ref{relation}) must be $c_{1:0}=0$ and hence for all $g\geq 3$ and $n\geq g-1$ any relation of type~(\ref{relation}) that is invariant under the action of $S_n$ permuting the marked points is trivial.

Now consider relations of type~(\ref{relation}) that are not invariant under the action of $S_n$. As the symmetrisation is trivial, the only non-zero coefficients in the relation are the coefficients of the classes not fixed under the action of $S_n$, namely, $c_{K_j}$ for $j=1,\dots,n$ and $c_{i:S}$ for $S\ne \emptyset,\{n+1\}, \{1,\dots, n\},\{1,\dots,n+1\}$. We now use the test surfaces up to a possible relabelling of the points $p_1,\dots,p_n$ to show that any relation~(\ref{relation}) in only these coefficients is necessarily trivial.

Test surface $S^f_{h:k}$ for $1\leq h\leq g-2$ and $k=2,\dots,n$ gives
\begin{equation}\label{NSfhk}
c_{h:S}=-c_{h+1:S}
\end{equation}
for any $S\subset\{1,\dots,n\}$ with $1\leq|S|\leq n-1$. Further, test surface $S^e_{h:1}$ for $2\leq h\leq g-2$ gives
\begin{equation}\label{NSeh1}
(2h-1)c_{K_j}=c_{h:\{j\}}
\end{equation}
and for $h=g-1$ gives
\begin{equation}\label{NSehgm1}
(2g-4)c_{K_j}=c_{g-1:\{j\}}
\end{equation}
for $j=1,\dots,n$. Equations~(\ref{NSfhk}),(\ref{NSeh1}) and (\ref{NSehgm1}) then give 
\begin{equation*}
c_{K_j}=c_{i:\{j\}}=0
\end{equation*}
for $j=1,\dots,n$ and $1\leq i\leq g-1$. Test surface $S^e_{h:2}$ for $1\leq h\leq g-1$ gives
\begin{equation*}
c_{h:S}=c_{0:S}
\end{equation*}
for $S\subset \{1,\dots,n\}$ and $|S|=2$, which combined with~(\ref{NSfhk}) gives $c_{i:S}=0$ for $S\subset \{1,\dots,n\}$ and $|S|=2$. Test surface $S^e_{h:k}$ for $0\leq h\leq g-1$ and $k=3,\dots,n$ gives $c_{i:S}=0$ in all remaining cases.
\end{proof}

\Bdry*

\begin{proof}
Consider the first projection
\begin{equation*}
\widehat{\Delta}_{1:\emptyset}=\overline{\mathcal{M}}_{g-1,n+1}\times \overline{\mathcal{M}}_{1,1}\longrightarrow \overline{\mathcal{M}}_{g-1,n+1}.
\end{equation*}
Pulling back the infinitely many extremal divisors of Proposition~\ref{Mullane} we obtain by Proposition~\ref{prod} infinitely many extremal divisors on $\widehat{\Delta}_{1:\emptyset}$ for $g\geq3$ and $n\geq g-1$. Propositions \ref{higherextremal} and \ref{injective} complete the proof.
\end{proof}

\section{Principal strata}\label{principal}

\Principal*

\begin{proof}
Proceed by induction. Assume that $\varphi_{j+1}{}_*\overline{\PPP}(1^{2g-2})$ is rigid and extremal. 
If $[(\varphi_j)_*\overline{\PPP}(1^{2g-2})]$ is not extremal then it can be expressed as
\begin{equation*}
[(\varphi_j)_*\overline{\PPP}(1^{2g-2})]=\sum c_i [V_i]
\end{equation*}
for $c_i>0$, $V_i$ irreducible with class not proportional to $[(\varphi_j)_*\overline{\PPP}(1^{2g-2})]$.
Pushing forward this equation by $\pi_k:\overline{\mathcal{M}}_{g,2g-j-2}\longrightarrow \overline{\mathcal{M}}_{g,2g-j-3}$ forgetting the $k$th marked point for $k=1,\dots,2g-j-2$ we obtain 
\begin{equation*}
(\pi_k)_*[(\varphi_j)_*\overline{\PPP}(1^{2g-2})]= [(\varphi_{j+1})_*\overline{\PPP}(1^{2g-2})]    =\sum c_i (\pi_k)_*[V_i]
\end{equation*}
for each $k=1,...,2g-j-2$.

As the LHS is non-zero, for a fixed $k$ there is at least one $V_i$ such that $(\pi_k)_*[V_i]$ is non-zero. Further, as the LHS is extremal, $(\pi_k)_*[V_i]$ is necessarily a positive multiple of $[(\varphi_{j+1})_*\overline{\PPP}(1^{2g-2})]$. As this cycle is rigid, $V_i$ must be supported on $(\pi_k)^{-1}(\varphi_{j+1}{}_*\overline{\PPP}(1^{2g-2}))$ and hence $(\pi_{k'})_*[V_i]$ is non-zero for any other $k'$. This argument for each $k'$ yields $V_i$ is supported in the intersection of $(\pi_k)^{-1}(\varphi_{j+1}{}_*\overline{\PPP}(1^{2g-2}) )$ for $k=1,...,2g-j-2$. In particular, any $2g-j-3$ points in a general element of $V_i$ are distinct points in a hyperplane section of the canonical embedding, hence all $2g-j-2$ points must be distinct points in a hyperplane section of the canonical embedding and $V_i$ is supported on $(\varphi_j)_*\overline{\PPP}(1^{2g-2})$ and hence is a positive multiple of $[(\varphi_j)_*\overline{\PPP}(1^{2g-2})]$ providing a contradiction. 

Hence $[\varphi_{j}{}_*\overline{\PPP}(1^{2g-2})]$ is extremal if $[\varphi_{j+1}{}_*\overline{\PPP}(1^{2g-2})]$ is rigid and extremal. Further, if $[(\varphi_j)_*\overline{\PPP}(1^{2g-2})]$ is extremal but not rigid, then
\begin{equation*}
[(\varphi_j)_*\overline{\PPP}(1^{2g-2})]=c[V]
\end{equation*}
for $c>0$ and $V$ not supported on $(\varphi_j)_*\overline{\PPP}(1^{2g-2})$. The above argument provides a contradiction. 

The base case of the inductive argument $j=g-2$ is the divisorial case presented in Proposition~\ref{FV}
\end{proof}

\section{Meromorphic strata}\label{meromorphic}
The strategy employed to show the principal strata are rigid and extremal can be applied to the meromorphic strata with an alteration for some lower genus cases. We provide the inductive argument as the following series of propositions separating the more involved lower genus cases.

\begin{prop}\label{mero1}
For $g\geq 3$ and $j=0,\dots,g-1$ the cycle $[\varphi_j{}_*\overline{\PPP}(d_1,d_2,d_3,1^{2g-3})]$ is extremal and rigid in $\mbox{Eff}^{\hspace{0.1cm}g-j}(\overline{\mathcal{M}}_{g,2g-j})$, where $\varphi_j:\overline{\mathcal{M}}_{g,2g}\longrightarrow\overline{\mathcal{M}}_{g,2g-j}$ forgets the last $j$ points with $d_1+d_2+d_3=1$, $\sum_{d_i<0}d_i\leq-2$ and some $d_i=1$ if $g=3$.
\end{prop}

\begin{proof} Proceed again by induction. Assume $[(\varphi_{j+1})_*\overline{\PPP}(d_1,d_2,d_3,1^{2g-3})]$ is rigid and extremal. 
If $[(\varphi_j)_*\overline{\PPP}(d_1,d_2,d_3,1^{2g-3})]$ is not extremal then it can be expressed as
\begin{equation*}
[(\varphi_j)_*\overline{\PPP}(d_1,d_2,d_3,1^{2g-3})]=\sum c_i [V_i]
\end{equation*}
for $c_i>0$, $V_i$ irreducible with class not proportional to $[(\varphi_j)_*\overline{\PPP}(d_1,d_2,d_3,1^{2g-3})]$. Pushing forward this equation under $\pi_k:\overline{\mathcal{M}}_{g,2g-j}\longrightarrow \overline{\mathcal{M}}_{g,2g-j-1}$ forgetting the $k$th point we obtain 
\begin{equation*}
(\pi_k)_*[(\varphi_j)_*\overline{\PPP}(d_1,d_2,d_3,1^{2g-3})]= [(\varphi_{j+1})_*\overline{\PPP}(d_1,d_2,d_3,1^{2g-3})]    =\sum c_i (\pi_k)_*[V_i]
\end{equation*}
for each $k=4,...,2g-j$ for $g\geq 3$. Without loss of generality assume that $d_3=1$ when $g=3$. Then the equation will hold in the $g=3$ case for $k=3,\dots, 6-j$. 

As the LHS is non-zero, for a fixed $k$ there is at least one $V_i$ such that $(\pi_k)_*[V_i]$ is non-zero. Further, as the LHS is extremal, $(\pi_k)_*[V_i]$ is necessarily a positive multiple of $[(\varphi_{j+1})_*\overline{\PPP}(d_1,d_2,d_3,1^{2g-3})]$. But as this cycle is rigid, $V_i$ must be supported on 
\begin{equation*}
(\pi_k)^{-1}(\varphi_{j+1}{}_*\overline{\PPP}(d_1,d_2,d_3,1^{2g-3}))
\end{equation*}
and hence $(\pi_{k'})_*[V_i]$ is non-zero for any other $k'=4,...,2g-j$ for $g\geq 4$ or $k'=3,\dots 6-j$ for $g=3$. This argument for each $k'$ yields $V_i$ is supported in the intersection of $(\pi_k)^{-1}(\varphi_{j+1}{}_*\overline{\PPP}(d_1,d_2,d_3,1^{2g-3}) )$ for $k=4,...,2g-j$ for $g\geq 4$ or $k=3,\dots 6-j$ for $g=3$. A general element of $V_i$ is hence of the form $[C,p_1,...,p_{2g-j}]\in \mathcal{M}_{g,2g-j}$ with
\begin{equation*}
d_1p_1+d_2p_2+d_3p_3+\sum_{i=4,i\ne k}^{2g-j}p_i +\sum_{i=1}^{j+1}q_i \sim K_C
\end{equation*}
for some $q_i$ with $k=4,...,2g-j$. But this implies that for $g\geq4$ the $p_i$ for $i=4,...,2g-j$ are all at least pairwise distinct and hence distinct. Similarly for $g=3$ the $p_i$ for $i=3,\dots,6-j$ are all at least pairwise distinct and hence distinct. 

Hence in this case we have $V_i$ is supported on $(\varphi_{j})_*\overline{\PPP}(d_1,d_2,d_3,1^{2g-3})$ and $[V_i]$ is a positive multiple of $[(\varphi_j)_*\overline{\PPP}(1^{2g-2})]$ providing a contradiction. 

Hence if $(\varphi_{j+1})_*\overline{\PPP}(d_1,d_2,d_3,1^{2g-3})$ is rigid and extremal then $(\varphi_{j})_*\overline{\PPP}(d_1,d_2,d_3,1^{2g-3})$ is rigid and extremal. The base case for the inductive argument is the divisorial case $j=g-1$ presented in Proposition~\ref{Mullane}.
\end{proof}

\begin{prop}\label{mero2}
For $g=3$ and $j=0,1,2$ the cycle $[\varphi_j{}_*\overline{\PPP}(d_1,d_2,d_3,1^{3})]$ is extremal and rigid in $\mbox{Eff}^{\hspace{0.1cm}3-j}(\overline{\mathcal{M}}_{g,6-j})$, where $\varphi_j:\overline{\mathcal{M}}_{3,6}\longrightarrow\overline{\mathcal{M}}_{3,6-j}$ forgets the last $j$ points with $d_1+d_2+d_3=1$, $\sum_{d_i<0}d_i\leq-2$.
\end{prop}

\begin{proof} The case where some $d_i=1$ is covered by Proposition~\ref{mero1}. Assume $d_i\ne 1$ and without loss of generality assume $d_1\geq 2$. Again, we proceed by induction. Assume $[(\varphi_{j+1})_*\overline{\PPP}(d_1,d_2,d_3,1^{3})]$ is rigid and extremal. If $[(\varphi_j)_*\overline{\PPP}(d_1,d_2,d_3,1^{3})]$ is not extremal then it can be expressed as
\begin{equation*}
[(\varphi_j)_*\overline{\PPP}(d_1,d_2,d_3,1^{3})]=\sum c_i [V_i]
\end{equation*}
for $c_i>0$, $V_i$ irreducible with class not proportional to $[(\varphi_j)_*\overline{\PPP}(d_1,d_2,d_3,1^{3})]$. Pushing forward this equation we obtain 
\begin{equation*}
(\pi_k)_*[(\varphi_j)_*\overline{\PPP}(d_1,d_2,d_3,1^{3})]= [(\varphi_{j+1})_*\overline{\PPP}(d_1,d_2,d_3,1^{3})]    =\sum c_i (\pi_k)_*[V_i]
\end{equation*}
for each $k=4,...,6-j$. As the LHS is non-zero, this implies for a fixed $k$, there is at least one $V_i$ such that $(\pi_k)_*[V_i]$ is non-zero. Further, as the LHS is extremal, $(\pi_k)_*[V_i]$ is necessarily a positive multiple of $[(\varphi_{j+1})_*\overline{\PPP}(d_1,d_2,d_3,1^{3})]$. But as this cycle is rigid, $V_i$ must be supported on 
\begin{equation*}
(\pi_k)^{-1}(\varphi_{j+1}{}_*\overline{\PPP}(d_1,d_2,d_3,1^{3}))
\end{equation*}
and hence $(\pi_{k'})_*[V_i]$ is non-zero for any other $k'=4,...,6-j$. This argument for each $k'$ yields $V_i$ is supported in the intersection of $(\pi_k)^{-1}(\varphi_{j+1}{}_*\overline{\PPP}(d_1,d_2,d_3,1^{3}) )$ for $k=4,...,6-j$. 

Hence for $j=0$ a general element of $V_i$ is of the form $[C,p_1,\dots,p_{6}]\in \mathcal{M}_{3,6}$ with
\begin{equation*}
d_1p_1+d_2p_2+d_3p_3+\sum_{i=4,i\ne k}^{6}p_i +q_k \sim K_C
\end{equation*}
for some $q_k$ for each $k=4,5,6$. But this implies that the $p_i$ for $i=4,5,6$ are all at least pairwise distinct and hence distinct. Hence $V_i$ is supported on $\overline{\PPP}(d_1,d_2,d_3,1^{3})$ providing a contradiction and showing if $[(\varphi_{1})_*\overline{\PPP}(d_1,d_2,d_3,1^{3})]$ is rigid and extremal then $[\overline{\PPP}(d_1,d_2,d_3,1^{3})]$ is rigid and extremal.

In the remaining case $j=1$, $V_i$ is supported in the intersection of $(\pi_k)^{-1}(\varphi_{2}{}_*\overline{\PPP}(d_1,d_2,d_3,1^{3}) )$ for $k=4,5$. In this case there are two possible candidates for where the irreducible cycle $V_i$ is supported. The cycle $V_i$ is supported on either $\varphi_1{}_*\overline{\PPP}(d_1,d_2,d_3,1^{3})$, or on the cycle 
\begin{equation*}
X:=\varphi_1^*(\varphi_2{}_*\overline{\PPP}(d_1,d_2,d_3,1^{3}))\cdot \delta_{0:\{4,5\}}
\end{equation*}
which can also be described as
\begin{equation*}
X:=\left\{ [C,p_1,p_2,p_3,q]\cup_{q=x}[\PP^1,x,p_4,p_5]\in \overline{\mathcal{M}}_{3,5}\hspace{0.2cm}\big|\hspace{0.2cm} [C,p_1,p_2,p_3,q]\in D^4_{d_1,d_2,d_3,1^3}      \right\}
\end{equation*}
where $[\PP^1,x,p_4,p_5]$ is a rational curve marked at three distinct points and 
\begin{equation*}
[D^4_{d_1,d_2,d_3,1^3}]=\frac{1}{2}\varphi_2{}_*[\overline{\PPP}(d_1,d_2,d_3,1^{3})].
\end{equation*}
The irreducibility of $X$ follows from the irreducibility of $\overline{\PPP}(d_1,d_2,d_3,1^{3})$. Hence if $V_i$ is supported on $X$ then $[V_i]$ is proportional to $[X]$ and
\begin{equation*}
\pi_1{}_*[V_i]=  e\delta_{0:\{3,4\}}
\end{equation*}
for some $e>0$.  

As the cycle $[(\varphi_1)_*\overline{\PPP}(d_1,d_2,d_3,1^{3})]-c_i[V_i]$ is an effective codimension two cycle in $\overline{\mathcal{M}}_{3,5}$, by pushing down under the morphism that forgets the first marked point we obtain the effective divisor class 
\begin{equation*}
\pi_1{}_*\left([(\varphi_1)_*\overline{\PPP}(d_1,d_2,d_3,1^{3})]-c_i[V_i]\right)=D^4_{d_2,d_3,1^3,d_1}-c_i e\delta_{0:\{3,4\}}.
\end{equation*}
However, by Proposition~\ref{covering} and Proposition~\ref{intg3} we observe
\begin{equation*}
B^4_{\kappa,1,-1}\cdot (D^4_{d_2,d_3,1^3,d_3}-c_i e\delta_{0:\{3,4\}})=0-24d_2^2d_3^2c_ie<0,
\end{equation*}
for $\kappa=(d_2,d_3,1^3,d_1)$, which contradicts the moving curve $B^4_{\kappa,1,-1}$ introduced in \S\ref{movingcurves} having non-negative intersection with all effective divisors. Hence $V_i$ is not supported on $X$ and must be supported on $\varphi_{1}{}_*\overline{\PPP}(d_1,d_2,d_3,1^{3})$. Hence $[\varphi_{1}{}_*\overline{\PPP}(d_1,d_2,d_3,1^{3})]$ is rigid and extremal if $[\varphi_{2}{}_*\overline{\PPP}(d_1,d_2,d_3,1^{3})]$ is rigid and extremal. The base case for the inductive argument is the divisorial case $j=2$ presented in Proposition~\ref{Mullane}.
\end{proof}

\begin{prop}\label{mero3}
For $g= 2$, the cycle $[\overline{\PPP}(h,-h,1,1)]$ for $h\geq2$ is extremal and rigid in $\mbox{Eff}^{\hspace{0.1cm}2}(\overline{\mathcal{M}}_{2,4})$.
\end{prop}

\begin{proof}
$[\pi_k{}_*\overline{\PPP}(h,-h,1,1)]$ is rigid and extremal for $k=3,4$ by Proposition~\ref{Mullane}. 
If $[\overline{\PPP}(h,-h,1,1)]$ is not extremal then it can be expressed as
\begin{equation*}
[\overline{\PPP}(h,-h,1,1)]=\sum c_i [V_i]
\end{equation*}
for $c_i>0$, $V_i$ irreducible with class not proportional to $[\overline{\PPP}(h,-h,1,1)]$. Pushing forward this equation we obtain 
\begin{equation*}
(\pi_k)_*[\overline{\PPP}(h,-h,1,1)]= [(\pi_k)_*\overline{\PPP}(h,-h,1,1)]    =\sum c_i (\pi_k)_*[V_i]
\end{equation*}
for $k=3,4$. However, this implies there is some $V_i$ such that $\pi_4{}_*[V_i]=[\pi_4{}_*V_i]\ne 0$. But as $[\pi_4{}_*\overline{\PPP}(h,-h,1,1)]$ is extremal (Proposition~\ref{Mullane}), $[\pi_4{}_*V_i]$ must be a positive multiple of $[\pi_4{}_*\overline{\PPP}(h,-h,1,1)]$. Further, as $[\pi_4{}_*\overline{\PPP}(h,-h,1,1)]$ is rigid, $V_i$ must be supported on $\pi_4^{-1}(\pi_4{}_*\overline{\PPP}(h,-h,1,1))$.

Hence $\pi_3{}_*[V_i]\ne 0$ and the same argument yields $V_i$ must be supported on $\pi_3^{-1}(\pi_3{}_*\overline{\PPP}(h,-h,1,1))$. The intersection 
\begin{equation*}
\pi_3^{-1}(\pi_3{}_*\overline{\PPP}(h,-h,1,1))\cap \pi_4^{-1}(\pi_4{}_*\overline{\PPP}(h,-h,1,1))
\end{equation*}
has two irreducible components. $V_i$ is either supported on $\overline{\PPP}(h,-h,1,1)$ or
\begin{equation*}
X:=\pi_4^*(\pi_4{}_*\overline{\PPP}(h,-h,1,1))\cdot \delta_{0:\{3,4\}}
\end{equation*}
which can also be described as
\begin{equation*}
X:=\left\{ [C,p_1,p_2,q]\cup_{q=x}[\PP^1,x,p_3,p_4]\in \overline{\mathcal{M}}_{2,4}\hspace{0.2cm}\big|\hspace{0.2cm} [C,p_1,p_2,q]\in D^3_{h,-h,1^2}      \right\}
\end{equation*}
where $[\PP^1,x,p_3,p_4]$ is a rational curve marked at three distinct points and 
\begin{equation*}
D^3_{h,-h,1^2}=\varphi_1{}_*\overline{\PPP}(h,-h,1,1)=\pi_4{}_*\overline{\PPP}(h,-h,1,1).
\end{equation*}
The irreducibility of $X$ follows from the irreducibility of $\overline{\PPP}(h,-h,1,1)$. Hence if $V_i$ is supported on $X$ then $[V_i]$ is proportional to $[X]$ and  
\begin{equation*}
\pi_1{}_*[V_i]=  e\delta_{0:\{2,3\}}
\end{equation*}
for some $e>0$.  

As the cycle $[\overline{\PPP}(h,-h,1,1)]-c_i[V_i]$ is effective, by pushing down under the morphism that forgets the first marked point we obtain the effective class 
\begin{equation*}
\pi_1{}_*\left([\overline{\PPP}(h,-h,1,1))]-c_i[V_i]\right)=D^3_{-h,1,1,h}-c_i e\delta_{0:\{2,3\}}.
\end{equation*}
However, by Proposition~\ref{covering} and Proposition~\ref{intg2} 
\begin{equation*}
B^3_{\kappa,1,-1}\cdot (D^3_{-h,1,1,h}-c_i e\delta_{0:\{2,3\}})=0-8h^2e<0,
\end{equation*}
for $\kappa=(-h,1,1,h)$ which contradicts the moving curve $B^3_{\kappa,1,-1}$ introduced in \S\ref{movingcurves} having non-negative intersection with all effective divisors. Hence $V_i$ is not supported on $X$ and must be supported on $\overline{\PPP}(h,-h,1,1)$ providing a contradiction with the given effective decomposition. Hence $[\overline{\PPP}(h,-h,1,1)]$ is rigid and extremal.
\end{proof}

We record the previous three propositions as the following theorem.

\Meromorphic*

This immediately gives the following corollary on the structure of the effective cones.

\Cone*

\begin{proof}
The rigid and extremal cycles presented in Theorem~\ref{meroextremal} have non-proportional classes as the pushforwards
\begin{equation*}
(\varphi_{g-j-1}){}_*[(\varphi_j)_*\overline{\PPP}(d_1,d_2,d_3,1^{2g-3})]=[(\varphi_{g-1})_*\overline{\PPP}(d_1,d_2,d_3,1^{2g-3})]=\frac{1}{(g-1)!}D^{g+1}_{d_1,d_2,d_3,1^{2g-3}}
\end{equation*}
have non-proportional classes as divisors by Proposition~\ref{Mullane}. Hence we have infinitely many extremal rays for $k=n-g$.

To extend the result, fix $k$ and pullback these classes under the forgetful morphism $\varphi:\Mgnb\longrightarrow\overline{\mathcal{M}}_{g,g+k}$. All previous arguments hold for the pullbacks.
\end{proof}

\section{Extremal cycles in genus one}\label{genusone}
In this section we examine the genus one case. In this case the meromorphic strata of canonical divisors have codimension one and to produce higher codimension cycles we intersect the pullbacks of strata under forgetful morphisms.

\begin{defn}
Set $m\geq 1$ and let $\underline{d}^j=(d_1^j,\dots,d_{n-m+1}^j)$ for $j=1,\dots,m$ be distinct non-zero integer partitions of zero. Then
\begin{equation*}
X(\underline{d}^1,\dots,\underline{d}^m):=\left\{[E,p_1,\dots,p_n]\in\mathcal{M}_{1,n}\hspace{0.15cm}\big| \hspace{0.15cm}\OO_E\left(d^j_{n-m+1}p_{n-m+j}+\sum_{i=1}^{n-m}d^j_ip_i \right)\sim \OO_E \text{ for $j=1,\dots,m$}\right\}
\end{equation*}
or alternatively
\begin{equation*}
X(\underline{d}^1,\dots,\underline{d}^m):=\left\{[E,p_1,\dots,p_n]\in\mathcal{M}_{1,n}\hspace{0.15cm}\big| \hspace{0.15cm}[E,p_1,\dots,p_{n-m},p_{n-m+j}]\in \PPP(\underline{d}^j)   \right\}
\end{equation*}
has codimension-$m$ in $\mathcal{M}_{1,n}$ with closure in $\overline{\mathcal{M}}_{1,n}$ denoted $\overline{X}(\underline{d}^1,\dots,\underline{d}^m)$. 
\end{defn}
We now specialise to the subvarieties of interest to us.

\Irred*

\begin{proof}
By forgetting the last $m-1$ points of $\overline{X}(\underline{d}^1,\dots,\underline{d}^m)$ in $\overline{\mathcal{M}}_{1,n}$ we obtain $\overline{\PPP}(\underline{d}^1)$, which by~\cite{Boissy} is irreducible. But for every 
\begin{equation*}
[E,p_1,\dots,p_{n-m+1}]\in \PPP(\underline{d}^1),
\end{equation*}
by the group law there is a unique 
\begin{equation*}
p_{n-m+j}=-\sum_{i=1}^{n-m}d^j_ip_i
\end{equation*}
which will in general be distinct from $p_i$ for $i=1,\dots, p_{n-m+j-1}$, hence 
\begin{equation*}
[E,p_1,\dots,p_{n-m},p_{n-m+j}]\in {\PPP}(\underline{d}^j)
\end{equation*}
for $j=2,\dots,m$. 

Since $\PPP(\underline{d}^1)$ is irreducible, the locus $X(\underline{d}^1,\dots,\underline{d}^m)$ is irreducible.
\end{proof}

\Genusone*

\begin{proof}
If $[\overline{X}(\underline{d}^1,\dots,\underline{d}^m)]$ is not extremal then it can be expressed as
\begin{equation*}
[\overline{X}(\underline{d}^1,\dots,\underline{d}^m)]=\sum c_i [V_i]
\end{equation*}
for $c_i>0$, $V_i$ irreducible with class not proportional to $[\overline{X}(\underline{d}^1,\dots,\underline{d}^m)]$. Let $\vartheta_j=\varphi_{\{n-m+1,\dots \widehat{n-m+j},\dots,n\}}$ and $\vartheta_{j,k}=\varphi_{\{n-m+1,\dots,n\}\setminus\{n-m+j,n-m+k\}}$ for $j\ne k$, that is, 
\begin{equation*}
\vartheta_j:\overline{\mathcal{M}}_{1,n}\longrightarrow \overline{\mathcal{M}}_{1,n-m+1}
\end{equation*}
forgets all but the marked points $1,\dots,n-m,n-m+j$ for $j=1,\dots,m$ and 
\begin{equation*}
\vartheta_{j,k}:\overline{\mathcal{M}}_{1,n}\longrightarrow \overline{\mathcal{M}}_{1,n-m+2}
\end{equation*}
forgets all but the marked points $1,\dots,n-m,n-m+j,n-m+k$ for $\{j,k\}\subset\{1,\dots,m\}$. Let $\vartheta_k^j$ be the intermediate map that forgets just one point such that $\vartheta_j=\vartheta_k^j\circ \vartheta_{j,k}$. Pushing forward under $\vartheta_j$ we obtain
\begin{equation*}
\vartheta_j{}_*[\overline{X}(\underline{d}^1,\dots,\underline{d}^m)]=c[\overline{\PPP}(\underline{d}^j)]=\sum c_i \vartheta_j{}_*[V_i]
\end{equation*}
where $c=1$ for $j=1$ and $c=(d^1_{n-m+1})^2$ for $j=2,\dots,m$. For fixed $j$ this implies there must be some $i$ such that $\vartheta_j{}_*[V_i]\ne0$. But as $[\overline{\PPP}(\underline{d}^j)]$ is an extremal divisor $\vartheta_j{}_*[V_i]$ must be a positive multiple of $[\overline{\PPP}(\underline{d}^j)]$ and further as $[\overline{\PPP}(\underline{d}^j)]$ is rigid, $\vartheta_j{}_*[V_i]$ must be supported in $\overline{\PPP}(\underline{d}^j)$ and hence $V_i$ is contained in $\vartheta_j^{-1}\overline{\PPP}(\underline{d}^j)$.

Further, this implies $\vartheta_{j,k}{}_*[V_i]\ne 0$ and hence $\vartheta_{j,k}{}_*V_i$ is contained in $(\vartheta_k^j)^{-1}\overline{\PPP}(\underline{d}^j)$. Hence $\vartheta_k{}_*[V_i]=(\vartheta_j^k)_*[\vartheta_{j,k}{}_*[V_i]]\ne0$ and the same argument shows $V_i$ is contained in $\vartheta_k^{-1}\overline{\PPP}(\underline{d}^k)$ and $\vartheta_{j,k}{}_*V_i$ is contained in $(\vartheta_j^k)^{-1}\overline{\PPP}(\underline{d}^k)$.

But this gives $\vartheta_{j,k}{}_*V_i$ is contained in $(\vartheta_k^j)^{-1}\overline{\PPP}(\underline{d}^j)\cap(\vartheta_j^k)^{-1}\overline{\PPP}(\underline{d}^k)$ for any choice of $j,k$ which implies $V_i$ is supported in $\overline{X}(\underline{d}^1,\dots,\underline{d}^m)$ providing a contradiction. 
\end{proof}
This immediately gives the following corollary on the structure of the effective cones.

\Genusonecor*

\begin{proof}
The rigid and extremal cycles presented in Theorem~\ref{thm:genusone} have non-proportional classes. Observe
\begin{equation*}
\vartheta_1{}_*[\overline{X}(\underline{d}^1,\dots,\underline{d}^k)]=[\overline{\PPP}(\underline{d}^1)]=D^{n-k+1}_{\underline{d}^1}
\end{equation*}
and
\begin{equation*}
\vartheta_j{}_*[\overline{X}(\underline{d}^1,\dots,\underline{d}^k)]=(d^1_{n-k+1})^2[\overline{\PPP}(\underline{d}^j)]=(d^1_{n-k+1})^2D^{n-k+1}_{\underline{d}^j}.
\end{equation*}
for $j=2,\dots,k$. Which provide non-proportional divisor classes for $[\overline{X}(\underline{d}^1,\dots,\underline{d}^k)]$ with distinct $\underline{d}^j$ by Theorem~\ref{ChenCoskun}. 
\end{proof}


\bibliographystyle{plain}
\bibliography{base}
\end{document}